\documentclass[11pt,a4paper]{article}
\usepackage{amsmath,amsthm,graphicx,rotating}
\usepackage{amssymb,bbm}
\usepackage{authblk}
\usepackage[utf8]{inputenc}
\usepackage[english]{babel}
\usepackage{xcolor}
\usepackage{hyperref}
\usepackage{comment}
\hypersetup{
    colorlinks=true,
    linkcolor=blue,
    filecolor=magenta,      
    urlcolor=cyan,
    }
\usepackage{ulem}
\newtheorem{theorem}{Theorem}[section]
\newtheorem{lemma}[theorem]{Lemma}
\newtheorem{corollary}[theorem]{Corollary}
\newtheorem{proposition}[theorem]{Proposition}

\newtheorem{remark}[theorem]{Remark}
\newtheorem{definition}[theorem]{Definition}
\newtheorem{definition-proposition}[theorem]{Definition-Proposition}
\newtheorem{hypothesis}[theorem]{Hypothesis}
\newtheorem{example}[theorem]{Example}

\numberwithin{equation}{section}

\newcommand\Tr{\mathrm{Tr}}
\newcommand\esp{\mathbb E}

\newcommand\limN{\underset{N \rightarrow \infty}\longrightarrow}
\newcommand\Nlim{\underset{N \rightarrow \infty}\lim}
\newcommand\eqa{\begin{eqnarray}}
\newcommand\qea{\end{eqnarray}}
\newcommand\eq{\begin{eqnarray*}}
\newcommand\qe{\end{eqnarray*}}
\newcommand\eps{\varepsilon}
\newcommand\etc{,\ldots ,}

\newcommand\one{\mathbbm{1}}
\newcommand\mbf{\mathbf}
\newcommand\mcal{\mathcal}
\newcommand\mbb{\mathbb}
\newcommand\mrm{\mathrm}

\newcommand*{\email}[1]{\href{mailto:#1}{\nolinkurl{#1}} } 

\title{The $\mathfrak S_k$-circular limit of random tensor flattenings}

\author[1]{St\'ephane Dartois}
\author[2]{Camille Male}
\author[3]{Ion Nechita}
\affil[1]{Université Paris-Saclay, CEA, List, Palaiseau, F-91120, France\\

\email{stephane.dartois@cea.fr}
}
\affil[2]{Institut de Mathématiques de Bordeaux, Université de Bordeaux, CNRS\\
351 Cours de la
Libération, 33400 Talence, France\\

\email{camille.male@math.u-bordeaux.fr}}
\affil[3]{Laboratoire de Physique Théorique\\
Université de Toulouse, CNRS, UPS\\
Toulouse, France\\

\email{ion.nechita@univ-tlse3.fr}}

\begin{document}
\date{}
\maketitle

\begin{abstract}

The tensor flattenings appear naturally in quantum information when one produces a density matrix by partially tracing the degrees of freedom of a pure quantum state. In this paper, we study the joint $^*$-distribution of the flattenings of large random tensors under mild assumptions, in the sense of  free probability theory. We show the convergence toward an operator-valued circular system with amalgamation on permutation group algebras for which we describe the covariance structure. As an application we describe the law of large random density matrix of bosonic quantum states.\\
\end{abstract}

Primary 15B52, 46L54; Keywords: 
Free Probability, Large Random Matrices, Random Tensors, Quantum Information, Bosonic quantum states

\section*{Acknowledgements} This research was funded in part by the Australian Research Council grant DE210101323 of Stephane Dartois and the French National Research Agency (ANR) under the projects STARS ANR-20-CE40-0008 and Esquisses ANR-20-CE47-0014-01.

\tableofcontents

\section{Introduction and presentation of the problem}

Let $F$ be an $N$-dimensional complex vector space given with a basis, and let $k\geq 1$ be a fixed integer. Thanks to the coordinates in this basis, we represent a tensor of order $2k$ on $F$ as a multi-indexed vector in $(\mbb C^N)^{\otimes 2k}$, such as
	$$M_N =\big( m(i_1\etc i_{2k})\big)_{i_1\etc i_{2k} \in [N]} \in (\mbb C^N)^{\otimes 2k},$$
where $[N]:=\{1\etc N\}$. Splitting the $k$ first and $k$ last indices, a tensor $M_N$ is canonically associated to a matrix of $\mrm{M}_{N^k}(\mbb C)$ that represents the endomorphism of $(\mbb C^N)^{\otimes k}$
	$$ M_{N,id} :=\big( m(\mbf i,\mbf j)\big)_{\mbf i, \mbf j \in [N]^{k}} \in \mrm{End}\big( ( \mbb C^N)^{\otimes k} \big) \cong\mrm{M}_{N^k}(\mbb C).$$
	
	  Furthermore, we can also represent a tensor differently by shuffling the role of the indices, producing different matrices here called ``flattenings" and sometimes ``matricization" or ``unfolding" of the initial tensor. More precisely, let us denote by $\mathfrak S_{2k}$ the symmetric group of order $2k$, that is the group of bijections of $[2k]$. 

\begin{definition}\label{def:flattening_from_symm_action} For any tensor $M_N \in (\mbb C^N)^{\otimes 2k}$ and any bijection $\sigma \in \mathfrak S_{2k}$, we define the element of $\mrm{M}_{N^k}(\mbb C) \sim\mrm{End}\big( (\mbb C^N)^k\big) $ 
	\eqa\label{PermutatedMatrix}
		M_{N,\sigma} :=  \Big (  M_N\big(i_{\sigma^{-1}(1)},\dots,i_{\sigma^{-1}(2k)}\big)  \Big)_{(i_1,\dots,i_k), (i_{k+1} \etc i_{2k})\in [N]^k},
	\qea
that is called a \emph{flattening} for short\footnote{In this article we only consider the here defined balanced flattenings, which are only a sub-family of all the flattenings of a tensor. In the context of this paper, we just call them flattenings.} of $M_N$.	
\end{definition}
For $k=1$, there are 2 permutations: $\mrm{id}$ and the transposition $(1,2)$, so there are 2 flattenings: the canonical one $M_{N,\mrm{id}}$ and its transpose $M_{N, (1,2)} = (M_{N,\mrm{id}})^\top$. For general order $k\geq 1$, there are up to $(2k)!$ different flattenings of a tensor. 

We work under the following assumptions.

\begin{hypothesis}\label{TensorModel}
    The tensor $M_N \in (\mbb C^N)^{\otimes 2k}$ has i.i.d.~entries distributed as a centered complex random variable $m_N$ having finite moments of all orders (i.e. $\esp\big[ |m_N|^\ell\big]<\infty$ for all $\ell$). Moreover, the following limits exist
    \begin{eqnarray}
    N^k \times \esp\big[ |m_N|^2\big]   \limN   c>0, &
    N^k \times \esp\big[ m_N^2\big]     \limN   c' \in \mbb C,
    \end{eqnarray}
    and for all non-negative integers $\ell_1,\ell_2$ such that $\ell_1+\ell_2>2$
    \begin{eqnarray}
    N^k \times \esp\big[ m_N^{\ell_1} \overline{m_N}^{\ell_2}\big] & \limN & 0.
    \end{eqnarray}
We call $(c,c')$ the parameter of $M_N$.
\end{hypothesis}

\begin{example}\label{Example}In item 2, we use the notation $x \leadsto \nu$ to mean that a complex random variable $x$ is distributed according to a probability distribution $\nu$.
 \begin{enumerate}
	\item Let $M_N$ be sampled according to the complex Ginibre ensemble, i.e. the entries of $\sqrt N M_N$ are distributed according to the standard complex Gaussian distribution. Then $M_N$ satisfies Hypothesis \ref{TensorModel} and its parameter is $(1,0)$. A real Ginibre ensemble also satisfies the hypothesis with parameter $(1,1)$.
	\item Let $p_N\in ]0,1], N\geq 1$ be a sequence of real numbers  and let $\mu$ be the distribution of a complex random variable with no atom in 0, and with finite moments of all orders. We denote the probability measure
	\eqa\label{ConvexDistr}
		\nu := (1-p_N) \delta_0 + p_N \mu.
	\qea
If $\mu$ is a Dirac mass, we assume $p_N< 1/2$. Let $M_N$ be  a random tensor with i.i.d. entries distributed as
	$\sigma^{-2}_N  ( x - \alpha \, p_N )$ with $x \leadsto \nu$, 
where for a variable $y\leadsto \mu$ we have set $\alpha = \esp[y]$, $\beta^2 = \esp[|y|^2]$ and $\sigma_N =N^k  p_N  (  \beta^2  -  |\alpha|^2p_N  )$.
If $N^kp_N$ tends to infinity, then $M_N$ satisfies Hypothesis \ref{TensorModel} and its parameter is $ ( \esp[|y|^2], \esp[y^2]\big)$ where $y\leadsto \mu$.
\end{enumerate}
\end{example}

\begin{remark} A variable sampled from $\nu$ defined in \eqref{ConvexDistr} is distributed according to  $\mu$ with probability $p_N$ and is equal to zero otherwise. We say that $\nu$ is a \emph{dilution} of $\mu$. The entries of $M_N$ are normalized to be centered and in order to get the announced parameter. The sequence  $p_N$ is allowed to converge to zero, provided the average number of entries different from the constant $\alpha p_N$ in each column of a flattenings of $M_N$ converges to infinity.
\end{remark}

In this article,   we consider the collection of all the flattenings of such a  random  tensor. We study the $^*$-distribution of this family in the sense of free probability (whose definition is recalled in Section \ref{Sec:results_presentation}). Under the above assumptions, we characterize the limit in simple terms  thanks to \emph{operator-valued free probability theory}, which is the non commutative analogue of conditional probability. Our main result stated in Theorem \ref{MainTh1}, establishes that the limit of the flattenings is an \emph{operator-valued circular collection}.

This extends freeness results of a random matrix and its transpose, which is known for unitarily invariant random matrices that converge in $^*$-distribution \cite{mingo2016freeness}, and more generally for "asymptotically unitarily invariant matrices" in the sense of traffics \cite{cebron2016traffic}, which includes matrices with i.i.d. entries. It also  In constrast with the results of these cited works, we show that in our more general setting freeness does not hold between all the flattenings.  Other studies of random tensors in the context of free probability include the asymptotic freeness of a Wishart matrix and its partial transpose \cite{popa2020partial}, asymptotic semicircularity for contracted Wigner-type tensors \cite{au2021spectral}. Our result also generalizes \cite[Theorem 4.13]{dartois2020joint} when we restrict our attention to square matrices. 

Our motivation for this paper comes from quantum information theory. Recall first that the singular values of a matrix $A_N$ are the square roots of the eigenvalues of the positive semidefinite matrix $A_NA_N^*$. For a random matrix $A_N$,  we call (averaged) \emph{empirical singular values distribution} the probability measure $ \esp\big[ \frac 1 N \sum_{i\in [N]} \delta_{s_i}\big]$, where $\delta_{s_i}$ denote the Dirac mass at the singular values $s_i$ of $A_N$. If $A_N$ is Hermitian, we call (averaged) \emph{empirical eigenvalues distribution} the probability measure $ \esp\big[ \frac 1 N \sum_{i\in [N]} \delta_{\lambda_i}\big]$, where the $\lambda_i$'s are the eigenvalues of $A_N$.  The symmetrized matrix  
	\eqa\label{Motiv}
	\sum_{\sigma, \sigma' \in \mathfrak S_{2k}} M_{N,\sigma}M_{N,\sigma'}
	\qea
is, up to normalization, the density matrix associated to the partial trace over a bipartition of a random quantum states of bosons.  Bosons are one of the two flavors (together with fermions) of undistiguinshable particles in nature.  Their quantum states must be invariant under exchange of particles.\\
The typical properties of such random states have not been studied in much details in the mathematical literature despite the fact that in many contexts (\textit{e.g.}~condensed matter) they are much more natural states to consider. An attempt at studying a model that is close in spirit can be found here \cite{dartois2022entanglement}. 

In the present paper, we describe the spectrum of the marginals of (un-normalized) random bosonic states in order to bound their limiting geometric bipartite entanglement. In fact a consequence of our work is that the largest Schmidt coefficients of such states before normalization is asymptotically bounded from below by $2$. This Schmidt coefficient is a well known measure of bipartite entanglement which directly relates to geometric entanglement \cite{wei2003geometric}.\\
This is a first step in our work, as we plan to use the results of this paper to study the spectrum of the partial transpose of bosonic quantum states in the future. Indeed, we know from numerical simulations their spectrum is behaving very differently from the spectrum of the partial transpose of quantum states with no symmetry. \\
Moreover, we hope that this formalism could be used to study general permutational criteria of entanglement \cite{horodecki2002separability} of which the partial transpose and the realignment criteria are well known special cases, as indeed the permutational criteria can be encoded through the action of $\mathfrak{S}_{2k}$ on density matrices.\\
Finally, though our work sheds no new light on this specific aspect, we want to point out the importance of flattenings of quantum states to study entanglement. In fact, the separability of quantum states is equivalent to the vanishing of all the $2\times 2$ minors of some\footnote{Note that these are not the flattenings we consider here, as we only consider a subset of them that produce maps $(\mathbb{C}^{N})^{\otimes k}\rightarrow(\mathbb{C}^{N})^{\otimes k}$, while contraction maps are all the flattenings leading to maps$(\mathbb{C}^{N})^{\otimes 2k-1}\rightarrow\mathbb{C}^N$ } of its flattenings, called \emph{contraction maps} \cite{ottaviani2013introduction}.

Another motivation comes from data analysis, where information on a noisy data tensor is retrieved from the spectral properties of its flattenings for instance in the context of Multilinear Subspace Analysis (MSA) \cite{geng2009face}. Our work could give insight on the free independence properties of the flattenings of the noisy part of such data tensors. In particular, we allow the tensor to be \emph{diluted}, i.e. to have a majority of zero entries which is a regime appearing in \textit{e.g.} community detection \cite{pal2021community} or in MSA with (a lot of) missing values. \\
Finally, one of the most important recent use of flattenings of tensors is in the context of tensors PCA, where when they are used to initialize tensor powers algorithms drastically improve the detection threshold of this family of algorithms while having a better accuracy above the threshold. See \cite{richard2014statistical} for a foundational work on this topic (see \cite[section 3 and 6]{richard2014statistical}).\\

To conclude this section, let us discuss in more detail a simple situation that our result solve. Let $M_N$ be a random tensor satisfying Hypothesis \ref{TensorModel} with parameter $(c,c')$. It is known that, for each $\sigma \in \mathfrak S_{2k}$, the empirical eigenvalues distribution of  $M_{N,\sigma}M_{N,\sigma}^*$ converges to the \emph{Mar{\v{c}}henko-Pastur distribution} $\mrm{MP}_c=\mrm{MP}_{c,1}$ of variance $c$ and shape parameter $\lambda=1$ (since the matrices we consider are square). We recall that the $\mrm{MP}_{c,\lambda}$ distribution has the following form
$$ \mrm {MP}_{c,\lambda} = \max(1-\frac 1 \lambda,0) \delta_0 +  \frac{\sqrt{(b-x)(x-a)}}{2\pi c \lambda x} \one_{x\in (a,b)}\mrm dx,$$
 with $a = c(1-\sqrt \lambda)^2$, $b = c(1 + \sqrt \lambda)^2$. As expressed earlier, in this work we consider the $\lambda=1$ case.

More generally, it is known in free probability that each flattening converges in non-commutative distribution  to a so-called \emph{circular variable} \cite{mingo2017free} (see next section for the definitions). Furthermore, when $c'=0$ the transpose $M_{N,\sigma}^\top$ is also known from works of Mingo and Popa \cite{mingo2016freeness} as well as Cébron, Dahlqvist and the second named author of this paper \cite{cebron2016traffic} to be asymptotically free from $M_{N,\sigma}$, so $S^{\pm}_{N,\sigma}=1/\sqrt 2(M_{N,\sigma} \pm M_{N,\sigma}^\top)$ converges to a circular variable with same variance. Therefore, the empirical eigenvalues distribution  of $S^{\pm}_{N,\sigma} {S^{\pm}_{N,\sigma}}^*$ converges to a Mar{\v{c}}henko-Pastur distribution. Similarly the empirical eigenvalues distribution of $S^{\pm}_{N,\sigma} + {S^{\pm}_{N,\sigma}}^*$  converges to a semicircular distribution. 

One can also check (see Section \ref{Sec:ProofCov}) that the limits of the matrices are decorrelated, that is $\esp \big[ \frac 1 N \Tr M_{ N,\sigma} M_{ N,\sigma'}^*\big]\limN 0$ if $\sigma\neq \sigma'$. Hence it is natural to wonder if the collection of all matrices $M_{ N,\sigma}$ converges to free circular variables for $k\geq 2$. This is not true as shows the following consequence of our main result. 

\begin{corollary}\label{MainAppl} Let $k\geq 1$ be a fixed integer. Let $M_N$ be a random tensor satisfying Hypothesis \ref{TensorModel} with parameter $(c,c')$. We consider the three following random matrices
	$$S_{1,N} = \frac 1{ \sqrt{  (2k)! k!c}}   \sum_{\sigma \in \mathfrak S_{2k}} M_{N,\sigma}, \quad S_{2,N}=  \frac 1{ \sqrt{  (2k)! k!c}}   \sum_{\sigma \in \mathfrak S_{2k}}\mathrm{sg}(\sigma)\, M_{N,\sigma},$$
	$$ S_{3,N} = \frac 1{2 \sqrt{  (2k)! k!(c+\Re c')}}   \sum_{\sigma \in \mathfrak S_{2k}}\big( M_{N,\sigma}+M_{N,\sigma}^*\big),$$
	where $\mathrm{sg}(\sigma)$ is the signature of the permutation $\sigma$. We denote by $\delta_0$ the Dirac mass at zero, by $\mrm{MP}$ the Marchenko-Pastur distribution of shape parameter $\lambda=1$ and variance $1$, and by $\mrm{SC}$  the standard semicircular distribution. 
\begin{enumerate}
	\item  The empirical  eigenvalues distribution of  $S_{1,N} S_{1,N}^*$ and $S_{2,N} S_{2,N}^*$ converge to  the distribution $ (1-{k!}^{-1})\delta_0 + {k!}^{-1} \mrm{MP}.$
	
	\item The empirical eigenvalues distribution of $S_{3,N}$  	converges to the distribution $ (1-{k!}^{-1})\delta_0 + {k!}^{-1} \mrm{SC}$.
\end{enumerate}
\end{corollary}

If the matrices $M_{ N,\sigma}$, $\sigma \in \mathfrak S_{2k}$ were asymptotically free, the limits in the two first items of the above corollary should be Mar{\v{c}}henko-Pastur distributions, not dilutions of the  Mar{\v{c}}henko-Pastur as we observe when $k\geq 2$.  Note that the parameter ${k!}^{-1}$ of dilution depends only on $k$. For a diluted tensor as in the second item of Example \ref{Example}, the limit does not depend on the dilution parameter $p_N$ of the model.

The asymptotic relations between the matrices $M_{N,\sigma}$ that imply the lack of freeness are well explained thanks to operator-valued free probability theory.  

{\bf Organisation of the paper.} Section \ref{Sec:results_presentation} is dedicated to the presentation of our results and the use of these for the concrete computation of limiting laws of flattenings of several tensors. This part is mostly algebraic and free probabilistic in nature. In a first subsection \ref{Sec:OpValVar}, we recall notions from \cite{speicher1998combinatorial} relevant to our work. In the subsection \ref{subsec:proof_cor_mainapp} we state our main result and prove our main application to marginals of symmetric and anti-symmetric (un-normalized) quantum states: Corollary \ref{MainAppl}. \\

In Section \ref{Sec:ProofPrelim}, we prove technical lemmas that we make a repeated use of throughout the paper. In Section \ref{Sec:ProofCov}, we prove the convergence of the operator valued covariance of the flattenings. \\

The Section \ref{Sec:ProofCirc} is the most technical part of the paper and is devoted to the convergence of the family of flattenings to a $\mathfrak{S}_k$-circular system. It contains a good amount of combinatorics of hypergraphs and therefore recalls the needed definitions. It also recalls the method of the injective trace and details important examples that are used as anchors in the last part of the proof.

\section{The circular limit of flattenings}\label{Sec:results_presentation}

\subsection{Circular variables over the symmetric group}\label{Sec:OpValVar}

 We first recall the classical notion of large $N$ limit for random matrices in free probability \cite{mingo2017free,speicher1998combinatorial}. All random matrices under consideration are assumed to have entries with finite moments of all orders. We set $\Phi_N:= \esp\big[ \frac 1 {N^k} \Tr \, \cdot \, \big]$ the normalized expected trace on $\mrm{M}_{N^k}(\mbb C)$. The integer $k\geq  1$ is fixed as $N$ varies, and for an integer $n\geq1$, we set $[n]=\{1\etc n\}$.

\begin{definition}\label{Def:StarDistr}
\begin{enumerate}
	\item A $*$-probability is a couple $(\mcal A, \phi)$ where $\mcal A$ is a $*$-algebra and $\phi: \mcal A\to \mbb C$ is a unital positive linear, i.e. $\phi(1_{\mcal A} ) =1$ and  $\phi(aa^*)\geq 0$ for all $a\in \mcal A$.
	\item Let $\mbf M_N = (M_{N,j})_{j\in J}$ be a family of random matrices in $\mrm M_{N^k}(\mbb C)$, and let $\mbf m=(m_j)_{j\in J}$ be a family of elements in a $*$-probability space $(\mcal A, \phi)$. We say that $\mbf M_N$ converges in $*$-distribution to $\mbf m$ whenever 
	\begin{equation}\label{CVdistr}
	\Phi_N\Big[ M_{N,j_1}^{\varepsilon_1}  \cdots M_{N,j_L}^{\varepsilon_L}  \Big]\limN \phi\big[ m_{j_1}^{\varepsilon_1}  \cdots m_{j_L}^{\varepsilon_L}\big] \in \mbb C
	\end{equation}
for any $L\geq 1$, and any  $j_\ell \in J$, $\varepsilon_\ell \in \{1,*\}$, $\ell\in [L]$. 
\end{enumerate}
\end{definition}

To describe the limit when $\mbf M_N$ is the collection of flattenings of a tensor $M_N$ satisfying Hypothesis \ref{TensorModel}, it is actually much easier to consider also other quantities, involving the following definitions.

\begin{definition} \label{OVstructure}
\begin{enumerate}
	\item For any $\eta \in \mathfrak{S}_k$, let $U_{N,\eta}$ be the unitary matrix of size $N^k$ whose action on a simple tensor $v_1\otimes\cdots\otimes v_k \in (\mathbb{C}^N)^{\otimes k}$ is
    $$U_{N,\eta}(v_1\otimes\cdots\otimes v_k) := v_{\eta(1)}\otimes \cdots \otimes v_{\eta(k)}.$$
We denote by $\mbb C \mathfrak{S}_{N,k}$ the vector space spanned by all the matrices $U_{N,\eta}$ for $\eta$ in $\mathfrak{S}_k$. 
	\item Recalling the notation $\Phi_N = \esp\big[ \frac 1N \Tr \, \cdot \, \big]$, let $\mcal E_N$ be the linear map  defined, for any random matrix $ A_N \in \mrm M_{N^k}(\mbb C)$, by
\eq
    \mcal E_N ( A_N)=  \sum_{\eta \in \mathfrak{S}_k} \Phi_N\big[ A_N U_{N,\eta}^* \big]U_{N,\eta} \in  \mbb C\mathfrak{S}_{N,k}.
\qe
\end{enumerate}
\end{definition}

In Definition \ref{DefOpValDistr} below, we define a notion of convergence with respect to $\mcal E_N$, rather than $\Phi_N$. To introduce it, we first state basic properties. Firstly, one sees that $\eta\mapsto U_{N,\eta}$ is a \emph{representation} of $\mathfrak{S}_k$, namely $U_{N,\eta} U_{N,\eta'} = U_{N,\eta \eta'}$ and $U_{N,\eta}^*=U_{N,\eta^{-1}}$ for all $\eta, \eta' \in \mathfrak{S}_k$. The next statements are proved in Section \ref{Sec:ProofPrelim}.

\begin{lemma}\label{PrelimBasic1} \begin{enumerate}
	\item For all $\eta, \eta' \in \mathfrak{S}_k$, $\Phi_N[U_{N,\eta}]= N^{\# \eta-k}$, where $\# \eta$ is the number of cycles of $\eta$. 
	\item The matrices of $(U_{N,\eta})_{\eta\in  \mathfrak{S}_k}$ are linearly independent when $N\ge k$. 
\end{enumerate}
\end{lemma}

In particular, when $N\geq k$, any element $B_N$ of $\mbb C\mathfrak{S}_{N,k}$ has a unique decomposition
	$$B_N =   \sum_{\eta \in \mathfrak{S}_k} B_N(\eta) U_{N,\eta}.$$
We say that the $B_N(\eta)$'s are the coefficients of $B_N$. We can hence introduce the following notion of (coefficient-wise) convergence for a sequence of elements in $\mbb C\mathfrak{S}_{N,k}$.

\begin{definition} \begin{enumerate}
	\item The group algebra $\mbb C \mathfrak{S}_k$ of $\mathfrak{S}_k$ is the vector space with basis $(u_\eta)_{\eta \in \mathfrak{S}_k}$, endowed with the product induced by $u_\eta u_{\eta'} := u_{\eta \eta'}$ and the antilinear involution induced by $u_\eta^*:=u_{\eta^{-1}}$. Every $b\in \mbb C \mathfrak{S}_k$ has a unique decomposition $b = \sum_{\eta \in \mathfrak{S}_k} b(\eta) u_\eta$.
	\item Let $(B_N)_{N\geq 1}$ be a  sequence  of elements of $\mbb C\mathfrak{S}_{N,k}$ and let $b$ in $\mbb C \mathfrak{S}_k$. We say that $B_N$ converges to $b$ as $N\rightarrow \infty$ whenever $B_N(\eta) \limN b(\eta)$ for all $\eta \in \mathfrak{S}_k$, in which case we write $B_N \limN b$.
\end{enumerate}
\end{definition}
 The following lemmas state properties of $\mcal E_N$. They are proved in Section \ref{Sec:ProofPrelim}.

\begin{lemma}\label{PrelimBasic2} For all $B_N, B_N' \in  \mbb C\mathfrak S_{N,k}$ and $A_N\in \mrm M_{N^k}(\mbb C)$, we have 
	\begin{eqnarray}
		\mcal E_N ( B_NA_NB'_N) &=&B_N \mcal E_N ( A_N)B'_N,  \label{CE}\\
		\mcal E_N ( \mbb I_N) & = & \mbb I_N +o(1)
	\end{eqnarray}
where the $o(1)$ means a sequence of elements of $\mbb C\mathfrak S_{N,k}$ whose coefficients converge to zero. Moreover, if the coefficients of $ \mcal E_N ( A_N)$ are bounded, then 
	\begin{eqnarray}
		\Phi_N \big[ \mcal E_N( A_N ) \big] &=& \Phi_N[A_N] + o(1). \label{EtoTrace}
	\end{eqnarray}
\end{lemma} 
 
\begin{lemma}\label{lem:EN-cp}
For all $N$, the map $\mcal E_N$ is completely positive, i.e.~the map $\mrm{id}_{\mrm{M}_p(\mbb C)} \otimes \mcal E_N$ is positive for all integers $p\geq 1$.
\end{lemma}

We mention this complete positivity property since it is included in the definition of operator-valued probability space, although we do not use it; for a proof, see Section \ref{Sec:ProofPrelim}.

We can now recall the notion of large $N$ limit with respect to $\mcal E_N$. In the definition below, we restrict our attention to the case of amalgamation over $\mbb C\mathfrak{S}_k$, although it can be replaced by any finitely generated unital $^*$-algebra.

\begin{definition}\label{DefOpValDistr}\begin{enumerate}
	\item An operator-valued $^*$-probability space with amalgamation over $\mbb C \mathfrak{S}_k$ (called in short, a $\mathfrak{S}_k^*$-probability space) is a triplet of the form $(\mcal A, \mbb C \mathfrak{S}_k, \mcal E)$ where $\mcal A$ is a $^*$-algebra, $\mbb C \mathfrak{S}_k$ is a subalgebra of $\mcal A$, and $\mcal E: \mcal A\to \mbb C \mathfrak{S}_k$ is a conditional expectation, i.e. a completely positive unital linear map such that 
		$$ \mcal E(bab') = b \mcal E(a)b'.$$
	\item Let $\mbf M_N = (M_{N,j})_{j\in J}$ be a family of random matrices in $\mrm M_{N^k}(\mbb C)$, and let $\mbf m=(m_j)_{j\in J}$ be a family of elements in a $\mathfrak{S}_k^*$-probability space $(\mcal A, \mathfrak{S}_k, \mcal E)$. We say that $\mbf M_N$ converges in $\mathfrak{S}_k^*$-distribution to $\mbf m$ whenever 
	\begin{equation}\label{CVOPdistr}
	\mathcal E_N\Big[ M_{N,j_1}^{\varepsilon_1}U_{N,\eta_1} \cdots M_{N,j_L}^{\varepsilon_L}U_{N,\eta_L}\Big]\limN  \mathcal E\big[ m_{j_1}^{\varepsilon_1}u_{\eta_1} \cdots m_{j_L}^{\varepsilon_L}u_{\eta_L} \big] \in \mbb C \mathfrak{S}_k
	\end{equation}
for any $L\geq 1$.
\end{enumerate}
\end{definition}

\begin{remark}\label{RemarkOnOpValSpaces} \begin{enumerate}
    \item   Let $(\mcal A, \mathfrak{S}_k, \mcal E)$ be a $\mathfrak{S}_k^*$-probability space. We can canonically define a linear form $\phi$ on $\mcal A$ by setting $\phi(a)$ equal to the coefficient of the unit in $\mcal E(a)$, for all $a \in \mcal A$. This induces a $^*$-probability space $(\mcal A, \phi)$ such that $\mcal E = \sum_{\eta \in \mathfrak S_k} \phi( \, \cdot \, u_\eta^*) u_\eta$ and $\phi( \mcal E) = \phi$. Moreover, if  $\mbf M_N$ converges in $\mathfrak{S}_k^*$-distribution to $\mbf m$ in a space $(\mcal A,\mbb C \mathfrak{S}_k, \mcal E)$ then $\mbf M_N$ also converges in $^*$-distribution to $\mbf m$ in $(\mcal A, \phi)$. 
    \item Note that we have $\mcal E(a)^* = \mcal E(a^*)$ and $\phi(a)^*=\phi(a^*)$ for all $a\in \mcal A$. Indeed, let us set as usual $\Re (a) = \frac {a+a^*}2$ and $\Im(a) = \frac{a-a^*}{2i}$, so that $a =\Re (a) + i \Im(a) $. The linearity of $\mcal E$ implies that $\mcal E\big( \Re (a) \big)= \Re\big( \mcal E(a)\big)$ and $\mcal E\big( \Im (a) \big)= \Im\big( \mcal E(a)\big)$. Therefore, the anti-linearity of the adjoint gives as expected $ \mcal E(a)^* = \Re\big( \mcal E(a) \big) - \Im \big( \mcal E(a) \big) = \mcal E(a^*).$ The same proof shows the similar property for $\phi$.
\end{enumerate}
\end{remark}

The random matrices considered in this article are proved to converge to so-called \emph{$\mathfrak{S}_k$-circular} variables. To define this notion, we introduce a sequence of multi-linear maps called the operator-valued free cumulants. 

In the definition below, we denote by $\mrm{NC}(n)$ the set of non-crossing partitions of the interval $[n]=\{1, \ldots, n\}$, we set $\mrm{NC}=\sqcup_{n\geq 1}\mrm{NC}(n)$, and we describe a way to extend a family of linear maps into a collection of maps labeled by non-crossing partitions.

\begin{definition-proposition}\label{NotationNC} Given any sequence $( \mcal M_n)_{n\geq 1}$ of linear maps $\mcal M_n: \mcal A^n \to \mbb C\mathfrak{S}_k$, we define canonically the collection $(\mcal M_\xi)_{ \xi \in \mrm {NC}}$ as follows. Let $\xi$ be a non-crossing partition of $[n]$. There always exists an interval block $B=\{i,i+1 \etc i+n'-1\}$ of $\xi$, for some $i\in [n]$ and $n'\in [n-i+1]$. Assume that $B$ is the interval block  of smallest indices, namely $i=\mrm{min}\big\{j  \in [n]\, | \, \exists \, m'\in [n] \,  \,\mrm{s.t.} \, \{ j, j+1  \cdots , j+m'-1\} \in \xi \big\}$. Then for all $a_1 \etc a_n$ in $\mcal A$, we set, inductively on $n$,\eq
	\lefteqn{\mcal M_\xi(a_1\etc  a_n)}\\
	& =&  \mcal M_{\xi \setminus B}\big( a_1\etc a_{i-1} \mcal M_{n'}(a_i ,a_{i+1} \etc a_{i+n'-1}) ,a_{i+n'}\etc a_n\big),
\qe
where $\xi \setminus B \in \mrm{NC}(n-n')$ is obtained by removing the block $B$ and shifting the indices greater than $i+n'-1$, with the convention $\mcal M_{\{\emptyset\}}=\mrm{id}_{\mbb C\mathfrak{S}_k}$. This relation entirely characterizes the collection $(\mcal M_\xi)_{ \xi \in \mrm {NC}}$ in terms of $(\mcal M_n)_{n\geq 1}$.
\end{definition-proposition}

\begin{remark} For $k=1$, one has $\mbb C \mathfrak{S}_1 = \mbb C$ and the functions $\mcal M_\xi$ satisfy a trivial factorization,	\eqa\label{TrivialFacto}
		\mcal M_\xi(a_1\etc a_n) = \prod_{ B= \{i_1 < \cdots < i_{n'}\} \in \xi  }  \mcal M_{n'}(a_{i_1} \etc a_{i_{n'}}).
	\qea
Formula \eqref{TrivialFacto} is not valid for $k\geq 2$ if the maps $\mcal M_\xi$ are not $\mbb C \mathfrak{S}_k$-linear, e.g. $\mcal M_2(ab,a') \neq \mcal M_2(a,a')b$ for some $a,a'\in \mcal A$, $b\in \mbb C \mathfrak{S}_k$. The general definition of $\mcal M_\xi$ depends on the nesting structure of the blocks of $\xi$. 
\end{remark}

\begin{definition-proposition}\label{Def:Cumulants}   The $\mathfrak{S}_k$-free cumulants on $(\mcal A, \mbb C \mathfrak{S}_k, \mcal E)$ are the unique collection $(\mcal K_n)_{n\geq 1}$ of linear maps $\mcal K_n: \mcal A^n \to \mbb C\mathfrak{S}_k$ such that, for all $n\geq 1$ and $a_1\etc a_n$ in $\mcal A$,
	$$ \mcal E( a_1\cdots a_n) = \sum_{ \xi \in \mrm{NC}(n)} \mcal K_\xi(a_1\etc a_n).$$
	Each $\mcal K_\xi$ is a $\mathbb{C}\mathfrak{S}_k$-module map, that is	\eq
		\mcal K_\xi( a_1 \etc a_{i-1}, a_i b, a_{i+1} \etc a_n) & = & \mcal K_\xi( a_1 \etc a_{i-1}, a_i , ba_{i+1} \etc a_n)\\
		  \mcal K_\xi(b a_1, a_2, a_3 \etc a_{n-1},a_nb') & = & b \, \mcal K_\xi( a_1 \etc a_n) b' ,
	\qe
 for all $a_1\etc a_n \in \mcal A$ and all $b,b'\in  \mbb C \mathfrak{S}_k$.
\end{definition-proposition}

We can now set the central definitions used to describe our matrices.

\begin{definition}\label{def:CircularFreeness}Let $(\mcal A, \mbb C \mathfrak{S}_k, \mcal E)$ be a $\mathfrak{S}_k^*$-probability space.
\begin{enumerate}
	\item  A collection $\mbf m = (m_j)_{j\in J}$ of elements in $\mcal A$ is \emph{$\mathfrak{S}_k$-circular} whenever the following $\mathfrak{S}_k$-cumulants of order greater than two vanish:
	$$\mcal K_n( m_{j_1}^{\varepsilon_1}b_1,m_{j_2}^{\varepsilon_2}b_2,  \cdots , m_{j_L}^{\varepsilon_L}b_L) = 0,$$ 
for all $L\geq 3$ and all $j_\ell\in J, \varepsilon_\ell \in \{1,*\}, b_\ell\in \mbb C\mathfrak{S}_k$ with $ \ell\in [L]$.
	\item Let $  A_1\etc   A_n$ be ensembles of elements of $\mcal A$. The $A_i$'s are \emph{free over $\mathfrak{S}_k$} (or $\mathfrak{S}_k$-free) if and only if the mixed $\mathfrak{S}_k$-cumulants vanish:
$$ {\mcal K}_n(m_{1}^{\epsilon_1}b_1,m_{2}^{\epsilon_2}b_2,\ldots, m_{L}^{\epsilon_L}b_L)=0,$$
for all  $\varepsilon_\ell \in \{1,*\}, b_\ell\in \mbb C\mathfrak{S}_k, \ell\in [L]$, and for all $m_1\etc m_L$ in the union $A_1 \cup \dots \cup A_n$ but not all in a single set $A_i$ (i.e. $\exists \ell \neq \ell'$ such that $m_\ell \in A_i , m_{\ell'}\in A_{i'}$ and $i\neq i'$).  
	
\end{enumerate}
\end{definition}

 To conclude this section, we state two classical lemmas useful to comment our main result. First, let us make explicit how the law of a $\mathfrak{S}_k$-circular collection is determined by first two $\mathfrak{S}_k$-free cumulants. With notations of Definition \ref{def:CircularFreeness}, the first cumulant coincides with the conditional  expectation, namely  $\mcal K_1(m_j) = \mcal E(m_j), \forall j\in J$. We say that the collection is centered when $\mcal K_1(m_j) =0$ for all $j\in J$.

 Moreover, since $\mathcal{K}_2$ is a $\mathbb{C}\mathfrak{S}_k$-bimodule bilinear map, the data of the second $\mathfrak{S}_k$-free cumulants can be summed up to the data of the \emph{$\mathfrak{S}_k$-covariances}
	\eqa\label{Eq:DefSigmaCov}
		\mcal K_2(m_j^{\eps}b, m_{j'}^{\eps'}) =\mcal E(m_j^{\eps}\, b  \,  m_{j'}^{\eps'}) - \mcal E(m_j^{\eps}  )  \, b \, \mcal E(m_{j'}^{\eps'} ), 
	\qea
 for all $j,j'\in J, \eps, \eps'\in \{1,*\}, b\in  \mbb C \mathfrak{S}_k$. We emphasize that $\mcal K_2(m_j^*b, m_{j'}^*) = \mcal K_2( m_{j'} b^* ,m_{j})^*$, which is a direct consequence of Formula \eqref{Eq:DefSigmaCov} and of the property $\mcal E(a)^* = \mcal E(a^*)$ (proved in the second item of Remark \ref{RemarkOnOpValSpaces}). Therefore, we can assume $(\eps, \eps')\neq (*,*)$ without lose of information on the $\mathfrak{S}_k$-covariance structure in Equation \eqref{Eq:DefSigmaCov}. The following lemma is a direct consequence of the Definition \ref{def:CircularFreeness}.

 \begin{lemma}\label{Lem:NulCovToFree}
 Let $\mbf m$ be a centered $\mbb C\mathfrak{S}_k$-circular system and let  $  A_1\etc   A_L$ be ensembles of variables in $\mbf m$. The ensemble $A_\ell$'s are free over $\mathfrak S_k$ if and only if the variables of different ensembles have pairwise vanishing $\mathfrak{S}_k$-covariance, namely $\mcal E_N\big( m_1^{\eps_1} \, b \, m_{2}^{\eps_2}  \big) =0,$  for all $\eps_1,\eps_2\in \{1,*\}$ all $b\in \mbb C \mathfrak S_k$, and all $m_1\in A_\ell, m_2 \in A_{\ell'}$ such that $\ell\neq \ell'$. 
\end{lemma}

To conclude this definition section, we now compare $\mathfrak{S}_k$-freeness with the usual notion of freeness. For $k=1$, we have $\mbb C \mathfrak S_1 = \mbb C$. Hence the $\mbb C \mathfrak{S}_1$-free cumulants coincide with the ordinary free cumulants. In particular, a collection of $\mathfrak{S}_1$-circular variables is an ordinary circular system and $\mathfrak{S}_1$-free ensembles are free is the usual sense. 
Given a $ \mathfrak{S}_k$-circular collection $\mbf m$, the following lemma gives a criterion to prove that $\mbf m$ is an ordinary circular collection in $(\mcal A, \phi)$.

\begin{lemma}\label{Lem:CompareFreeness} Let $(\mcal A, \mathfrak{S}_k, \mcal E, \phi)$ as above and let $\mbf m$ be a $\mathfrak{S}_k$-circular collection. Then $\mbf m$ is circular in $(\mcal A, \phi)$ if $\mcal E(m_1^{\eps_1}  m_2^{\eps_2})$ is proportional to the unit $u_{\mrm{id}}$ 
for all $m_1, m_2$ in $\mbf m$ and all $\eps_1, \eps_2$ in $\{1,*\}$.
\end{lemma}

\begin{proof} Up to a centering, we can assume the collection is centered with respect to $\mcal E$ (the lemma is proved for centered variables, one easily deduces the result in the non-centered case). For all $n\geq 1$, we denote by $\mrm{NC}(n)$ the set of pair non-crossing partitions of $[n]$. Since the conditional expectation is proportional to the unit, by definition of $\phi$ we have $\mcal E(m_1^{\eps_1}  m_2^{\eps_2}) = \phi(m_1^{\eps_1}  m_2^{\eps_2}) u_{\mrm{id}}$ for all $m_1, m_2,\eps_1, \eps_2$. 
 Hence the trivial factorization \eqref{TrivialFacto} is valid for the $\mathfrak{S}_k$-cumulant functions $\mcal K_\xi$ on the variables and their adjoint. Hence we can write, for all $m_1\etc m_n$ in $\mbf m$  and all $\eps_1\etc \eps_n$ in $\{1,*\}$
	\eq
		\phi( m_1^{\eps_1} \cdots m_n^{\eps_n}) 
		& = & \phi\big( \mcal E( m_1^{\eps_1} \cdots m_n^{\eps_n})  \big)\\
		& = & \phi\Big( \sum_{ \xi \in \mrm{NC}_2 (n)} \mcal K_\xi(m_1^{\eps_1}, \cdots ,m_n^{\eps_n})    \Big)\\
		& = & \phi\Big( \sum_{ \xi \in \mrm{NC}_2 (n)}    \prod_{ \{i,j\} \in \xi} \phi(m_i^{\eps_i}  m_j^{\eps_j})u_{\mrm{id}} \Big)\\
		& = & \sum_{ \xi \in \mrm{NC} (n)}  \Big( \prod_{ \{i,j\} \in \xi}  \phi(m_i^{\eps_i}  m_j^{\eps_j}) \Big).
	\qe
By uniqueness of the cumulants (Proposition-Definition \eqref{Def:Cumulants} with $k=1$), the free cumulants of $\mbf m$ (with respect to $\phi$) greater than two vanish, so the collection $\mbf m$ is circular (with respect to $\phi$).
\end{proof}

\subsection{Main result and applications}\label{Sec:MainTh}

In this section we first state our main result (Theorem \ref{MainTh1}), together with some of its consequences exhibiting families of $\mathfrak S_k$-free and scalar-free flattenings. We then use those results to prove Corollary \ref{MainAppl}, which is one of the main applications (and motivations) of our work.

\subsubsection{Statement and corollaries}

We can now state our main result on the flattenings of a random tensor.

\begin{theorem}\label{MainTh1} Let $M_N$ be a random tensor satisfying Hypothesis \ref{TensorModel} with parameter $(c,c')$. Then the collection of flattenings of $M_N$ converges in $\mathfrak{S}_k^*$-distribution to a  centered $\mathfrak{S}_k$-circular family $\mbf m = (m_\sigma)_{\sigma \in \mathfrak S_{2k}}$, in some space $(\mcal A, \mbb C \mathfrak S_k, \mcal E)$. For all $\eta,\eta'\in \mathfrak{S}_k$, we denote $\eta\sqcup \eta'\in\mathfrak{S}_{2k}$ the permutation obtained by gluing the actions of $\eta$ and $\eta'$ on the first $k$ and the last $k$ elements of $[2k]$; it is formally defined by
\begin{equation}
    (\eta\sqcup \eta')(i):=\begin{cases} \eta(i) & \textrm{if  } i \in [k] \\
    \eta'(i-k)+k & \textrm{otherwise.}\end{cases}
\end{equation}
We also denote by $\tau \in \mathfrak{S}_{2k}$ the permutation swapping the first and the last $k$ elements of $[2k]$; it is given by 
	\eq
		\tau(i) & : & \left\{ \begin{array}{ccc}  i \in [k] & \mapsto & i+k \in [2k]\setminus [k] \\
									i \in [2k]\setminus [k] & \mapsto & i-k \in [k]
									\end{array}\right.
	\qe
The $\mathfrak{S}_k$-covariance of  $\mbf m$ is given by the following equalities: 
 $\forall \sigma, \sigma'\in \mathfrak S_{2k}, \forall \eta \in \mathfrak{S}_k$, 
\eqa
	\mcal E( m_{\sigma}u_\eta m_{\sigma'}^*) = \left\{ \begin{array}{cc} c u_{\eta'} & \mrm{ \ if \ } \exists \eta' \in \mathfrak{S}_k\textrm{ s.t. }   \sigma = (\eta' \sqcup \eta)\sigma' \\ 
 0 & \mrm{ \ otherwise,} \end{array}\right. \label{Formule1}\\ 
	\mcal E( m_{\sigma}^*u_\eta m_{\sigma'}) = \left\{ \begin{array}{cc} c u_{\eta'} & \mrm{ \ if \ }  \exists \eta' \in \mathfrak{S}_k \textrm{ s.t. }   \sigma =(\eta\sqcup \eta')\sigma'\\
 0 & \mrm{ \ otherwise,} \end{array}\right. \label{Formule2}\\
 \mcal E( m_{\sigma}u_\eta m_{\sigma'}) = \left\{ \begin{array}{cc} c' u_{\eta'}& \mrm{ \ if \ } \exists \eta' \in \mathfrak{S}_k \textrm{ s.t. } \sigma = \tau(\eta \sqcup \eta')\sigma'     \\
 0 & \mrm{ \ otherwise.} \end{array}\right. \label{Formule3}
\qea
\end{theorem}

 Let us first explain where the expression of the $\mathfrak S_k$-covariance comes from. First we show  that the limiting ordinary (i.e.~scalar) covariance of the collection $\mbf m$ in the above theorem is given by the simple relations: $\forall \sigma, \sigma'\in \mathfrak S_{2k}$,
	\eqa
		\Phi( m_{\sigma} m_{\sigma'}^*) = \left\{ \begin{array}{cc} c & \mrm{ \ if \ }   \sigma = \sigma' \\
 0 & \mrm{ \ otherwise,} \end{array}\right.  \quad
 		\Phi( m_{\sigma} m_{\sigma'}) = \left\{ \begin{array}{cc} c' & \mrm{ \ if \ }   \sigma = \tau\sigma' \\
 0 & \mrm{ \ otherwise.} \end{array}\right. \label{Formule0}
	\qea
This implies formulas \eqref{Formule1} and \eqref{Formule3} thanks to the following natural action of $\mbb C \mathfrak S_k$ on the flattenings.

\begin{lemma}\label{MainLemma} For any tensor $M_N\in (\mbb C^N)^{\otimes 2k}$ and any permutations $\sigma \in \mathfrak S_{2k}$, $\eta, \eta' \in \mathfrak{S}_k$, we have

$
	 U_{N,\eta}  M_{N,\sigma}U_{N,\eta'}^*= M_{N, (\eta \sqcup \eta') \sigma}.
$
Hence if the collection of flattenings $\mbf M_N$ of a tensor converges in $\mathfrak{S}_k^*$-distribution to some collection $\mbf m=(m_\sigma)_{\sigma \in \mathfrak S_k}$, we can always assume 
$ u_\eta m_\sigma u_{\eta'}^*=
m_{(\eta \sqcup \eta') \sigma}$. 
\end{lemma}

Lemma \ref{MainLemma} is proved in Section \ref{Sec:ProofPrelim}.  Section \ref{Sec:ProofCov} contains the details of the proof of the expressions of the $\mathfrak S_k$-covariance. 
 The main difficulty is to prove the $\mathfrak S_k$-circularity, which is the purpose of Section \ref{Sec:ProofCirc}, where we use the \emph{traffic} method in order to prove that the $\mcal E_N$-moments of the sequence of flattenings asymptotically satisfy a non-commutative Wick formula.

\medskip

In the rest of this subsection, we interpret the $\mathfrak S_k$-covariance thanks to group theory notions.
 Recall, that given $K$ a subgroup of a group $G$ and an element $g$ of $G$, the set $Kg:=\{kg\vert k\in K\}$ is called the right $K$-coset of $g$ in $G$. When the context is clear, we simply call it a $K$-coset. The set of $K$-cosets is denoted $K\backslash G$. Firstly, the permutations of $\mathfrak{S}_{2k}$ of the form $\eta\sqcup \eta'$ with  $\eta,\eta'\in \mathfrak{S}_k$, form  a subgroup, isomorphic to $\mathfrak{S}_{k}^2$, that we denote $\mathfrak{S}_{k,k}$. 
It is often referred as a \emph{Young subgroup}.  Its cosets are well-known and studied concepts of particular importance in the context of the representations of the permutation and general linear groups, see for instance \cite{hazewinkel2010algebras, james_1984}.

We now consider the permutation $\tau$. We denote by  $\langle \tau \rangle $ the subgroup  of $\mathfrak{S}_{2k}$ generated by $\tau$, and by $\mathfrak S_{k,k} \rtimes \langle \tau \rangle$ the subgroup of $\mathfrak{S}_{2k}$ generated by $\mathfrak{S}_{k,k}$ and $\tau$. This permutation has a natural action on matrices since one sees easily that $ M_{N, \tau \sigma} = M_{N,\sigma}^\top$ for all $\sigma \in \mathfrak{S}_{2k}$.
The structures of the subgroup $\mathfrak S_{k,k} \rtimes \langle \tau \rangle$ and the associated cosets are quite elementary: one sees that $\tau (\eta \sqcup \eta') = (\eta' \sqcup \eta) \tau$ for all $\eta, \eta'\in \mathfrak S_{k}$, and for all $s\in \mathfrak S_{k,k} \rtimes \langle \tau \rangle$, there is a unique $(s_1,s_2) \in \mathfrak{S}_{k,k}\times \langle \tau \rangle$ such that $ s = s_1s_2$. 
 Moreover, if $\sigma_\mcal O$ denotes a representative of a  $\mathfrak S_{k,k} \rtimes \langle \tau \rangle$-coset $\mcal O$, then $\mcal O$ is the union of the $\mathfrak S_{k,k}$-coset of $\sigma_\mcal O$ and the $\mathfrak S_{k,k}$-coset of $\tau\sigma_\mcal O$, and both of them are isomorphic to the $\mathfrak S_{k,k}$-coset of the identity. The group $\mathfrak S_{k,k} \rtimes \langle \tau \rangle$ is actually a semi-direct product of $\mathfrak{S}_{k,k}$ by $\langle \tau \rangle$, see \cite{scott1987group}.
 
 The next corollary describes the classes of asymptotic $\mathfrak S_k$-free matrices from a collection of flattenings. In the statement below, we say that ensembles of matrices are asymptotically $\mathfrak S_k$-free if they converge in $\mathfrak S_k^*$-distribution toward $\mathfrak S_k$-free ensembles.
 
\begin{corollary}\label{MainTh3} Let $M_N$ be a random tensor satisfying Hypothesis \ref{TensorModel} with parameter $(c,c')$.
\begin{enumerate}
	\item If $c'=0$, then the flattenings of $M_N$ indexed by different $\mathfrak{S}_{k,k}$-cosets are asymptotically $\mathfrak S_k$-free.
	\item If $c'\neq 0$, then the flattenings of $M_N$ indexed by different $\mathfrak S_{k,k} \rtimes \langle \tau \rangle$-cosets are asymptotically  $\mathfrak S_k$-free.
\end{enumerate}
In each case, different flattenings belonging to the same coset are not $\mathfrak S_k$-free.
\end{corollary}

\begin{proof} Let $\mbf m$ denotes the limit of $\mbf M_N$ in $\mathfrak S_{k}$-distribution.
By Theorem \ref{MainTh1},  $\mbf m$ is  $\mathfrak S_k$-circular collection and $\mathcal{K}_2(m_{\sigma_1}^{\epsilon_1}u_\eta,m_{\sigma_2}^{\epsilon_2})=0$ for all $\eta, \eps_1,\eps_2$ if the permutations $\sigma_1$ and $\sigma_2$ are in different cosets ($\mathfrak S_{k,k}$-cosets if $c'=0$ and $\mathfrak S_{k,k} \rtimes \langle \tau \rangle$-cosets if $c'\neq 0$). By Lemma \ref{Lem:NulCovToFree}, the vanishing of generalized covariances implies $\mathfrak S_k$-freeness. Moreover, if $\sigma$ and $\sigma'$ belong to the same coset, then by Theorem \ref{MainTh1} there exists $\eta \in \mathfrak S_{k}$ such that $\mcal E( m_\sigma u_\eta m_{\sigma'})$ or $\mcal E( m_\sigma u_\eta m_{\sigma'}^*)$ is non-zero, showing that the corresponding flattenings are not asymptotically $\mathfrak S_k$-free. 
\end{proof}
 
 In the following corollary and in next subsection, we show that some families are not only $\mathfrak S_k$-free, but actually free in the usual (scalar) sense. This is a stronger statement, which allows one to use the full machinery of free probability (over the scalars) to analyze the limit distributions of flattenings (as we did in Corollary \ref{MainAppl}).

\begin{corollary}\label{cor:Cfree-general} Let $M_N$ be a random tensor satisfying Hypothesis \ref{TensorModel} with parameter $(c,c')$.
Pick one element in each $\mathfrak S_{k,k} \rtimes \langle \tau \rangle$-coset to form a sub-collection $\widetilde{\mbf M}_N$ of $(2k)!(k!)^{-2}/2$ flattenings of $M_N$. Then $\widetilde{\mbf M}_N$ converges to a free circular system in the ordinary sense.  If moreover $c'=0$, then $\widetilde{\mbf M}_N \cup \widetilde{\mbf M}^\top_N$  converges to a  free circular system. \end{corollary}

\begin{proof} Theorem \ref{MainTh1} implies that $\widetilde{\mbf M}_N$ converges to a $\mathfrak S_k$-circular collection $\widetilde{\mbf m}$ that satisfies $\eps(m^\eps  m^{\eps'} ) =\phi(m^\eps_\sigma m^{\eps'}_{\sigma'}) u_{\mrm id}$ for all $m,m'$ in $\widetilde{\mbf m}$ and all $\eps,\eps'\in \{1,*\}$. By Lemma \ref{Lem:CompareFreeness}, the collection is an ordinary circular system. The reasoning is the same, under the assumption that $c'=0$, to prove the asymptotic circularity along with the transpose.
\end{proof}

\subsubsection{Construction of \texorpdfstring{$\mathfrak S_k$}{Sk}-free and free circular systems}

Let $M_N$ be a random tensor satisfying Hypothesis \ref{TensorModel}. The flattenings labeled by elements of the same $\mathfrak{S}_{k,k}$-coset are not asymptotically $\mathfrak S_k-$free and they do not converge to a circular system in the usual sense. But we can construct linear combination of these matrices with these properties.

For that purpose, we recall that a representation $(\rho, V)$ of a group $G$ is the data of a finite dimensional vector space $V$ and of a group morphism $\rho$ from $G$ to the space of endomorphisms of $V$. The character associated to a representation $\rho$ is the map
	$$\chi^{(\rho)} : \sigma \in G \mapsto \Tr \, \rho(\sigma).$$
A representation is irreducible if there is no proper subspace $\{0\} \subsetneq W \subsetneq V$ stable by $\rho(\sigma)$ for all $\sigma$ in $G$. The set $\mathrm{Irrep} \ \mathfrak S_k$ of irreducible representations of the symmetric group $\mathfrak S_k$ is well-known \cite{james_1984}. Irreducible representations are labelled by the so-called Young diagrams and the associated vector spaces are known as the Specht modules. 

We can now state a first corollary where we present asymptotically $\mathfrak S_k$-free collections of matrices build from flattenings in the same coset (which are therefore not $\mathfrak S_k$-free).

\begin{corollary}\label{Cor:FreeLinearCombi1}Let $M_N$ be a random tensor satisfying Hypothesis \ref{TensorModel} with parameter $(c,c')$ such that  $c'=0$ and let $\sigma$ be a fixed element of $\mathfrak S_{2k}$.  For any irreducible representation $\rho$ of $\mathfrak S_k$, let us consider
	$$S_{N,\rho} =  \sum_{\eta_1,\eta_2\in \mathfrak S_k}  b_\rho (\eta_1) \chi^{(\rho)}(\eta_2) M_{N,(\eta_1\sqcup \eta_2)\sigma},$$
where $b_\rho: \mathfrak S_k\to \mbb C$ is arbitrary.  Then the family $\{S_{N,\rho}\}_{\rho\in \mathrm{Irrep} \, \mathfrak S_k}$ converges to a $\mathfrak S_k$-free $\mathfrak S_k$-circular system. The same holds for the family of all matrices indexed by pairs $(\rho, \rho')$ of irreducible representations 
		$$S_{N,\rho, \rho'} =  \sum_{\eta_1,\eta_2\in \mathfrak S_k} \chi^{(\rho)}(\eta_1) \chi^{(\rho')}(\eta_2) M_{N,(\eta_1\sqcup \eta_2)\sigma}.$$
If moreover $b_\rho(\eta)=0$ for all  $\eta \neq \mrm{id}$ then the family $\{S_{N,\rho}\}_{\rho\in \mathrm{Irrep} \, \mathfrak S_k}$ converges in $^*$-distribution to a usual free circular system.
\end{corollary}

\begin{remark} The number of representation of the irreducible representations of $\mathfrak S_{k}$ is the number $\mrm{part}(k)$ of partitions of the integer $k$ as they index the conjugacy classes of $\mathfrak S_{k}$ (see \cite[Proposition 2.30]{fulton2013representation}). Moreover, the number of $\mathfrak S_{k,k}$-cosets of $\mathfrak S_{2k}$ is $\frac{(2k)!}{k!^2}$. One can pick for each representative $\sigma$ of a $\mathfrak S_{k,k}$-coset $\mcal O$ a family $S_{N,\rho, \rho'}$. Then the asymptotic $\mathfrak S_k$-freeness of flattenings labeled in different cosets and the above corollary proves that there is at least $\mrm{part}(k)^2\times \frac{(2k)!}{k!^2}$ matrices coupled with $M_N$, in the algebra generated by its flattenings, that converge to a $\mathbb{C}\mathfrak S_k$-free $\mathbb{C}\mathfrak S_k$-circular system. The same result for the matrices $S_{N,\rho} $ proves that there is at least $\mrm{part}(k)\times \frac{(2k)!}{k!^2}$ matrices coupled with $M_N$, in the algebra generated by its flattenings, that converge to a free circular system.
\end{remark}

To prove the corollary, let us first state a lemma.
\begin{lemma}\label{lem:free-with-coeff-SN} Let $M_N$ be a random tensor satisfying Hypothesis \ref{TensorModel} with parameter $(c,c')$ such that  $c'=0$, and let $\sigma$ be an element of $\mathfrak S_{2k}$. Consider two matrices of the form 
	\eq
		S_{N}& = &   \sum_{\eta_1,\eta_2\in \mathfrak S_k} a (\eta_1,\eta_2)M_{N,(\eta_1\sqcup \eta_2)\sigma},\\
		S_N' & = &   \sum_{\eta_1,\eta_2\in \mathfrak S_k} a' (\eta_1,\eta_2)M_{N,(\eta_1\sqcup \eta_2)\sigma},
	\qe
where the $a(\eta_1,\eta_2)$'s and the $a'(\eta_1,\eta_2)$'s are complex coefficients, and $\sigma$ is a given element of $\mathfrak S_{2k}$. Then the couple  $(S_N, S'_N)$ converges to a  $\mathfrak S_k$-free $\mathfrak S_k$-circular system  iff
  \begin{equation}\label{ConditionLinearComb}
    \sum_{\mu_1,\mu_2 \in \mathfrak{S}_k} a(\eta_1\mu_1 , \eta_2 \mu_2) \overline{a'(\mu_1 , \mu_2) }=0, \quad \, \forall \eta_1, \eta_2\in \mathfrak{S}_k.
\end{equation}
If moreover for all $\eta \neq \mrm{id}$ in $\mathfrak S_k$ we have
		\eqa\label{ConditionLinearComb2}
		 \sum_{\substack{\mu_1, \mu_2\in \mathfrak{S}_k}}
               a(\eta\mu_1,\mu_2)\overline{a(\mu_1,\mu_2)} = \sum_{\substack{\mu_1, \mu_2\in \mathfrak{S}_k}}
               a'(\eta\mu_1,\mu_2)\overline{a'(\mu_1,\mu_2)} = 0,
              	\qea
then the couple $(S_N, S'_N)$ converges to an ordinary circular system.
\end{lemma}

\begin{proof} The couple $(S_N, S'_N)$ is asymptotically $\mathfrak S_k$-circular since its entries are $\mathfrak S_k$-linear combination of the flattenings of $M_N$. Let $\mbf m = (m_\sigma)_{\sigma \in \mathfrak S_{2k}}$ be the limit of $\mbf M_N=  (M_{N,\sigma})_{\sigma \in \mathfrak S_{2k}}$, and denote the limits of  $S_N$ and $S'_N$ by
	\eq
		s :=   \sum_{\eta_1,\eta_2\in \mathfrak S_k} a (\eta_1,\eta_2)m_{(\eta_1\sqcup \eta_2)\sigma}, \quad 
		s'  :=  \sum_{\eta_1,\eta_2\in \mathfrak S_k} a' (\eta_1,\eta_2)m_{(\eta_1\sqcup \eta_2)\sigma}.
	\qe
 We shall characterize the $a (\eta_1,\eta_2)$'s and $a '(\eta_1,\eta_2)$'s such that $\mathcal{E}(su_\eta {s'}^{*})=0$. 
 By  \eqref{Formule1} we have
	\eq
            	\lefteqn{    \mathcal{E}(su_\eta {s'}^*)}\\
             & = &    \sum_{\substack{\eta_1,\eta_2\in \mathfrak{S}_k \\ \eta_1',\eta_2'\in \mathfrak{S}_k}} a(\eta_1, \eta_2)\overline{a'(\eta_1',\eta_2')} \mathcal{E}(m_{(\eta_1\sqcup\eta_2)\sigma }u_\eta m_{(\eta_1'\sqcup\eta_2')\sigma }^*)\\
               &= & \sum_{\substack{\eta_1,\eta_1'\in \mathfrak{S}_k \\ \eta_2,\eta_2'\in \mathfrak{S}_k\\ \eta' \in \mathfrak S_k}}
              a(\eta_1, \eta_2)\overline{a'(\eta_1',\eta_2')}
                 \delta_{(\eta_1\sqcup\eta_2),(\eta'\eta_1'\sqcup\eta \eta_2')} c u_{\eta'} \\
               & = & c \sum_{\eta' \in \mathfrak S_k} \Big(   \sum_{\eta_1',\eta_2'\in \mathfrak{S}_k}
               a(\eta'\eta_1', \eta\eta_2')\overline{a'(\eta_1',\eta_2')}
                \Big) 
               u_{\eta'}. 
	\qe
Therefore, $\mathcal{E}(su_\eta {s'}^*)=0$ iff each coefficient in front of $u_{\eta'}$ in the above expression of $ \mathcal{E}(su_\eta {s'}^*)$ is zero, which is the condition \eqref{ConditionLinearComb} of the statement. Moreover one sees with a similar computation that $\mathcal{E}(su_\eta {s'})$ is always zero thanks to the condition $c'=0$, hence \eqref{ConditionLinearComb} characterizes the $\mathfrak S_k$-freeness of $s$ and $s'$. 

Moreover, to characterize when $s$ and $s'$ converge to free circular variable, we use Lemma \ref{Lem:CompareFreeness}. Hence we shall prove that for all $m_1,m_2\in \{s,s'\}$ and all $\varepsilon_1, \varepsilon_2$ in $\{1,*\}$, the covariance $\mathcal{E}(m_1^{\varepsilon_1} m_2^{\varepsilon_2})$ is scalar. Since it is zero if $m_1 \neq m_2$ (by the above), and since the expression of $s$ and $s'$ are analogue, it remains to prove that the covariance vanishes for $m_1=m_2=s$. We have with the same computation as above
\eq
               \mathcal{E}(s  {s}^*) 
               & = & c \sum_{\eta' \in \mathfrak S_k} \Big(   \sum_{\eta_1',\eta_2'\in \mathfrak{S}_k}
                         a(\eta'\eta_1, \eta_2')\overline{a(\eta_1',\eta_2')}
                          \Big) 
                u_{\eta'}. 
	\qe
Hence $\mathcal{E}(s  {s}^*)$ is scalar iff for all $\eta'\neq \mrm{id}$ the coefficients in front of $u_{\eta'}$ in the above expression are zero, which is condition \eqref{ConditionLinearComb2}. Moreover, $\mathcal{E}(s  {s})$ is scalar, since the condition $c'=0$ and the same computation as above shows that it vanishes. We can hence use Lemma \ref{Lem:CompareFreeness}, proving the convergence of $S_N$ and $S_N'$ toward free circular variables.
\end{proof} 

The next proposition recalls an important property of irreducible representations of the symmetric group.

\begin{proposition}\label{prop:vanishing-convol-S_k}
Let $\{\chi^{\rho}\}_{\rho\in \mathrm{Irrep}(\mathfrak S_k)}$ be characters of irreducible representations of the symmetric group $\mathfrak S_k$, then assuming $\rho, \rho'$ are two different irreducible representations, 
\begin{equation}
    \sum_{\mu \in \mathfrak S_k}\chi^{(\rho)}(\eta \mu)\bar \chi^{(\rho')}(\mu)=0,\quad \forall \eta \in \mathfrak{S}_k.
\end{equation}
\end{proposition}
\begin{proof}
    If $\chi^{(\rho)}, \chi^{(\rho')}$ are characters of two irreducible representations of a finite group $G$, then their convolution writes \cite{fulton2013representation} 
    \begin{equation}
        \chi^{(\rho)}\star \chi^{(\rho')}=\delta_{\rho, \rho'}\frac{\lvert G\rvert}{\dim \rho }\chi^{(\rho)},
    \end{equation}
    where the convolution is defined as
    \begin{equation}
       (\chi^{(\rho)}\star \chi^{(\rho')})(h)=\sum_{g\in G} \chi^{(\rho)}(hg^{-1}) \chi^{(\rho')}(g).
    \end{equation}
    Specifying to the symmetric group $\mathfrak{S}_k$, using the fact that its characters are real valued class functions ($\chi^{(\rho)}(\mu)=\chi^{(\rho)}(\mu^{-1}), \, \chi^{(\rho)}(\mu)=\bar \chi^{(\rho)}(\mu)$)  to recognize that $$\sum_{\mu \in \mathfrak S_k}\chi^{(\rho)}(\eta \mu)\bar \chi^{(\rho')}(\mu)=(\chi^{(\rho)}\star\chi^{(\rho')})(\eta),$$ we can conclude.
\end{proof}

\begin{proof}[Proof of Corollary \ref{Cor:FreeLinearCombi1}]
The collection $(S_{N,\rho})_{\rho\in \mathrm{Irrep} \, \mathfrak S_k}$ converges to a $\mathfrak S_k$-circular system  since its entries are linear combinations of flattenings of $M_N$. Its is asymptotically $\mathfrak S_k$-free if each couple $(S_{N,\rho},S_{N,\rho'})$ in this collection is asymptotically $\mathfrak S_k$-free. Condition \eqref{ConditionLinearComb} of Lemma \ref{lem:free-with-coeff-SN} applied for this matrix reads
	$$ \sum_{\mu_1\in \mathfrak{S}_k} b_\rho( \eta_1 \mu_1) \overline{ b_{\rho'}( \mu_1) } \sum_{\mu_2\in \mathfrak{S}_k}  \chi^{(\rho)}(\eta_2\mu_2) \overline{\chi^{(\rho')}(\mu_2) } =0, \, \forall \eta_1, \eta_2\in \mathfrak{S}_k.$$
The above equality is satisfied thanks to Proposition \ref{prop:vanishing-convol-S_k}. Hence the family is asymptotically $\mathfrak S_k$-free. 

Moreover, Condition \eqref{ConditionLinearComb2}, reads
	$$\sum_{\mu_1\in \mathfrak{S}_k} \delta_{ \eta \mu_1, \mrm{id} }\delta_{ \mu_1, \mrm{id} } \sum_{\mu_2\in \mathfrak{S}_k}  \chi^{(\rho)}(\mu_2) \overline{\chi^{(\rho')}(\mu_2) }=0, \forall \eta\neq \mrm{id},$$
which is clearly true. Hence Lemma \ref{lem:free-with-coeff-SN} implies that the family is asymptotically free.

The same reasoning yields to the asymptotic $\mathfrak S_k$-freeness of the family $(S_{N,\rho,\rho'})_{\rho, \rho'\in \mathrm{Irrep} \, \mathfrak S_k}$.
\end{proof}

\begin{remark} After this section, and considering the results of Corollary \ref{MainAppl}, one would be inclined to study the laws of the following non-commutative random variables 
$$ \psi^{(\lambda)}:=\frac{\dim V_{(\lambda)}}{ (2k)!}\sum_{\sigma \in \mathfrak{S}_{2k}}\bar \chi^{(\lambda)}(\sigma)m_{\sigma}$$ 
where $\lambda$ labels an irreducible representation of $\mathfrak{S}_{2k}$ and $V_{(\lambda)}$ is the support of the corresponding representation. Note that this dimension is the number of Young tableaux of shape $\lambda$ \cite[Problem 4.47]{fulton2013representation}. If one introduces the projector
$$P^{(\lambda)}=\frac{\dim V_{(\lambda)}}{ (2k)!}\sum_{\sigma \in \mathfrak{S}_{2k}}\bar \chi^{(\lambda)}(\sigma) U_{N,\sigma},$$
then, according to \cite[section 2.4]{fulton2013representation}, 
$$P^{(\lambda)}(M_{N,\textrm{id}})=\frac{\dim V_{(\lambda)}}{ (2k)!}\sum_{\sigma \in \mathfrak{S}_{2k}}\bar \chi^{(\lambda)}(\sigma)M_{N,\sigma}$$
lives in the irreducible (up to multiplicity) sub-representation $V^{(\lambda)}$ of $(\mathbb{C}^N)^{\otimes 2 k}$. Therefore, $P^{(\lambda)}(M_{N,\textrm{id}})$ can model a parastatistics (random) quantum state \cite{parastatistics}. In the large $N$ limit, $P^{(\lambda)}(M_{N,\textrm{id}})$ converges to $\psi^{(\lambda)}$. Typical (entanglement) properties of such states have not been studied and it could be interesting to explore them further. We have tried to compute the limit law using our Theorem \ref{MainTh1} however, we were not able to reach a conclusion. Denoting $K_{\lambda}:=\frac{\dim V_{(\lambda)}}{ (2k)!}$, the $\mathfrak{S}_k$-covariance writes  
\begin{align*}
    &\mathcal{E}(\psi^{(\lambda)}u_\eta(\psi^{(\lambda)})^*)=K_{\lambda}^2\sum_{\sigma,\sigma'\in \mathfrak{S}_{2k}}\bar \chi^{(\lambda)}(\sigma)\chi^{(\lambda)}(\sigma')\mathcal{E}(m_{\sigma}u_\eta m_{\sigma'}^{*})\\
    &=K_{\lambda}\sum_{\eta'\in \mathfrak{S}_k} c \chi^{(\lambda)}(\eta'\sqcup\eta) u_{\eta'}.
\end{align*}
Our attempt at computing the moments of the $\psi^{(\lambda)}$'s leads to complicated combinations of Littlewood-Richardson coefficients for which we found no practical expression.
\end{remark}

\subsubsection{Examples: proof of Corollary \ref{MainAppl}}\label{subsec:proof_cor_mainapp}

We showcase the use of Theorem \ref{MainTh1} by proving Corollary \ref{MainAppl}. 
We recall that we consider three random matrices
	$$S_{1,N} = \frac 1{ \sqrt{  (2k)! k!c}}   \sum_{\sigma \in \mathfrak S_{2k}} M_{N,\sigma}, \quad S_{2,N}=  \frac 1{ \sqrt{  (2k)! k!c}}   \sum_{\sigma \in \mathfrak S_{2k}}\mathrm{sg}(\sigma)\, M_{N,\sigma},$$
	$$ S_{3,N} = \frac 1{2 \sqrt{  (2k)! k!(c+\Re c')}}   \sum_{\sigma \in \mathfrak S_{2k}}\big( M_{N,\sigma}+M_{N,\sigma}^*\big),$$
where $M_N$ satisfies Hypothesis \ref{TensorModel}. We address the question of computing the limit of empirical eigenvalues distribution of $S_{1,N}S_{1,N}^*$, $S_{2,N}S_{2,N}^*$ and of $S_{3,N}$. We start by computing the limit of the covariances with respect to $\mcal E_N$.

\begin{lemma}\label{InternalLemma} Let 
$s_1 =  \frac 1{ \sqrt{  (2k)! k!c}}   \sum_{\sigma \in \mathfrak S_{2k}} m_\sigma$, 
$s_2=\frac1{\sqrt{(2k)!k!c}}\sum_{\sigma \in \mathfrak S_{2k}}\mathrm{sg}(\sigma)\, m_\sigma$ and 
$s_3 = \frac 1{ \sqrt{  (2k)! k!2(c+\mathfrak R c')}}   \sum_{\sigma \in \mathfrak S_{2k}}\big( m_\sigma+m_\sigma^*\big)$, 
where $\mbf m$ is the limit of $\mbf M_N$. Then we have for all $\eta \in \mathfrak S_k$,
	 \eq
		\mcal E \big( s_1 u_\eta s_1^* )  =\mcal E \big( s_1^* u_\eta s_1 ) = \mcal E \big( s_3 u_\eta s_3) = (k!)^{-1}  \sum_{ \eta' \in  \mathfrak S_{k}} u_{\eta'},
	\qe
	\eq
		 \mathcal{E}(s_2u_{\eta}s_2^*)=\mathcal{E}(s_2^*u_{\eta}s_2)=(k!)^{-1}\mathrm{sg}(\eta)\, \sum_{\eta'\in \mathfrak{S}_k}\mathrm{sg}(\eta')\, u_{\eta'}.
	\qe
\end{lemma}

Note that in each case, the right hand side term is independent of $\eta$.

\begin{proof}[Proof of Lemma \ref{InternalLemma}] {\bf Case of $s_1$.} 
First note that for any $\eta$ in $\mathfrak S_k$, by Lemma \ref{MainLemma} we have
	$$s_1 u_\eta = \frac 1{ \sqrt{  (2k)! k!c}}   \sum_{\sigma \in \mathfrak S_{2k}} m_\sigma u_\eta = \frac 1{ \sqrt{  (2k)! k!c}}   \sum_{\sigma \in \mathfrak S_{2k}} m_{(\mathrm{id}\sqcup \eta^{-1})\sigma} = s_1,$$
thanks to a change of variable in the last equality. Similarly, using the equality $u_\eta m_\sigma = m_{(\eta\sqcup \mrm{id})\sigma}$ we get $u_\eta s_1 = s_1$ for all $\eta $ in $\mathfrak S_k$. Hence for all $\eta$ in $\mathfrak S_k$ we get
	\eq
		\mcal E(s_1u_\eta s_1^*) &= & \mcal E(s_1s_1^*) = \sum_{\eta' \in \mathfrak S_k} \Phi( s_1 s_1^* u_{\eta'}^*) u_{\eta'}\\
	& =& \Phi( s_1 s_1^* ) \times \sum_{\eta' \in \mathfrak S_k}  u_{\eta'}\\
	& = & \frac1{ {(2k)! k! c }}\sum_{\sigma, \sigma' \in \mathfrak S_{2k}}  \Phi(m_\sigma m_{\sigma'}^* ) \times \sum_{\eta' \in \mathfrak S_k}  u_{\eta'}
	\qe
But by \eqref{Formule0}, we have $\Phi(m_\sigma m_{\sigma'}^* )  = c \delta_{\sigma, \sigma'}$, therefore
	\eq
	\mcal E(s_1u_\eta s_1^*) = \frac{ c (2k)!}{ (2k)!k!c}  \times \sum_{\eta' \in \mathfrak S_k}  u_{\eta'}
	= \frac 1 {k!} \sum_{\eta' \in \mathfrak S_k}  u_{\eta'}
	\qe
Moreover, thanks to the same arguments, we have
	\eq
		\mcal E( s_1^*u_\eta s_1) & = & \mcal E( s_1^* s_1) = \sum_{\eta' \in \mathfrak S_k} \Phi( s_1^* s_1) u_{\eta'}
		= \mcal E( s_1u_\eta s_1^*)
	\qe
This concludes the computation of the $\mathfrak S_k$-covariance for $s_1$.

{\bf Case of $s_2$.} We now have the following computation,  for any $\eta$ in $\mathfrak S_k$ 
\begin{align*}
    s_2u_{\eta}=\frac1{\sqrt{(2k)! k! c }}\sum_{\sigma\in \mathfrak{S}_{2k}}\mathrm{sg}(\sigma)m_\sigma u_\eta =\frac1{\sqrt{(2k)! k! c }}\sum_{\sigma\in \mathfrak{S}_{2k}}\mathrm{sg}(\sigma)m_{(\mathrm{id}\sqcup \eta^{-1})\sigma}.
\end{align*}
Recalling that $\mathrm{sg}(\alpha \beta)=\mathrm{sg}(\alpha)\mathrm{sg}(\beta)$ and $\mathrm{sg}(\mrm{id})=1$, we have 
\begin{align*}
    s_2u_{\eta} = \frac1{\sqrt{(2k)! k! c }}\sum_{\sigma\in \mathfrak{S}_{2k}}\mathrm{sg}(\eta)\mathrm{sg}((\mrm{id}\sqcup \eta^{-1})\sigma)m_{(\mrm{id}\sqcup \eta^{-1})\sigma}=\mathrm{sg}(\eta)s_2.
\end{align*}
Similarly we have $u_\eta s_2 = \mathrm{sg}(\eta^{-1})s_2 =\mathrm{sg}(\eta)s_2  $. Hence we get, for all $\eta$ in $\mathfrak S_k$
	\eq
		\mcal E(s_2u_\eta s_2^*) &= &\mathrm{sg}(\eta)  \mcal E(s_2s_2^*) = \mathrm{sg}(\eta) \sum_{\eta' \in \mathfrak S_k} \Phi( s_2 s_2^* u_{\eta'}^*) u_{\eta'}\\
	& =&\mathrm{sg}(\eta)  \Phi( s_2 s_2^* ) \times \sum_{\eta' \in \mathfrak S_k}  \mathrm{sg}(\eta')  u_{\eta'}\\
	& = &  \mathrm{sg}(\eta) \sum_{\sigma, \sigma' \in \mathfrak S_{2k}} \frac{ \mathrm{sg}(\sigma)  \mathrm{sg}(\sigma') \Phi(m_\sigma m_{\sigma'}^* )}{ {(2k)! k! c }} \times \sum_{\eta' \in \mathfrak S_k}  \mathrm{sg}(\eta')  u_{\eta'}\\
	& = & \frac{ \mathrm{sg}(\eta)}{k!}   \sum_{\eta' \in \mathfrak S_k}  \mathrm{sg}(\eta') u_{\eta'}.
	\qe
{\bf Case of $s_3$.} Since $s_3$ is proportional to $s_1+s_1^*$, we have $s_3u_\eta = u_\eta s_3 = s_3$. Hence we have for any $\eta$ in $\mathfrak S_k$
	\eq
		\mcal E( s_3 u_{\eta} s_3) & = & \mcal E( s_3  s_3) = \Phi(s_3s_3^*) \times \sum_{\eta' \in \mathfrak S_k} u_\eta\\
		& = & \frac1{ {(2k)! k! 2(c+\Re c') }}\sum_{\sigma, \sigma' \in \mathfrak S_{2k}}    \Phi\big( (m_\sigma+ m_{\sigma}^*) (m_{\sigma'}+ m_{\sigma'}^*)\big) \times \sum_{\eta' \in \mathfrak S_k}  u_{\eta'}
	\qe
By \eqref{Formule0}, we have $\Phi(m_\sigma^* m_{\sigma'} )=\Phi(m_\sigma m_{\sigma'}^* )  = c \delta_{\sigma, \sigma'}$ and $\Phi(m_\sigma m_{\sigma'} ) = \overline{ \Phi(m_\sigma^* m_{\sigma'}^*)} = c' \delta_{\sigma, \tau\sigma'}$, so 
\eq
		\mcal E( s_3 u_{\eta} s_3) & = & \frac1{ {(2k)! k! 2(c+\Re c') }}\sum_{\sigma, \sigma' \in \mathfrak S_{2k}}    2( c + \mathfrak R c') \times \sum_{\eta' \in \mathfrak S_k}  u_{\eta'}\\
		& = & \frac 1 {k!} \sum_{\eta' \in \mathfrak S_k}  u_{\eta'}.
	\qe

\end{proof}

We can now prove Corollary \ref{MainAppl}. We start with the computation for the two first matrices and use the notations $S_N$ and $s$ to designate either $S_{1,N}$ and $s_1$, or $S_{2,N}$ and $s_2$. By the moment method, it is sufficient to show that the limit of $\Phi_N \big( (S_NS_N^*)^n \big)$ is the $n$-th moment of the expected limiting distribution, for all $n\geq 1$. Since $\mbf m$ is $\mathfrak S_k$-circular and $s$ is a linear combination of $\mbf m$, $s$ is also $\mathfrak S_k$-circular. Hence we have for all $n\geq 1$
	\eq
		\mcal E\big( (ss^*)^n ) & = & \sum_{ \xi \in \mrm{NC}_2(2n)} \mcal K_\xi( \underbrace{s,s^* \etc s,s^*}_{2n}).
	\qe
Recall that by definition of $\mcal K_\xi$, if $B$ denotes the first interval block of $\xi$,
	\eq
		\lefteqn{\mcal K_\xi(s,s^* \etc s,s^*)}\\
		& = &   \left\{ \begin{array}{cc}  \mathcal  K_{\xi\setminus B} ( \dots, s \mcal E \big(s^* s \big), s^*, \dots  ),  & \quad \mrm{if \ } B=\{2i,2i+1\},\\
		 \mathcal K_{\xi\setminus B} (\dots , s^*\mcal E \big( ss^*\big),  s, \dots  ),  & \quad \mrm{if \ } B=\{2i-1,2i\},\end{array} \right.
	\qe
where $\xi\setminus B$ is the partition obtained from $\xi$ by removing the block $B$ and shifting indices above. If $n=0$, the formula is valid with the convention $\mathcal K_{\{\emptyset\}}=\mrm{id}$.

We now consider the first case $s=s_1$. By Lemma \ref{InternalLemma}, we get with the same disjunction of cases as above
	\eq
		 \mcal K_\xi(s,s^* \etc s,s^*)& = &  \left\{ \begin{array}{c} (k!)^{-1}  \sum_{ \eta' \in  \mathfrak S_{k}} \mathcal K_{\xi\setminus B} (\dots , s u_{\eta'}, s^*, \dots  ) \mrm{ \ or}\\
		 (k!)^{-1}  \sum_{ \eta' \in  \mathfrak S_{k}} \mathcal K_{\xi\setminus B} (\dots , s^*u_{\eta'},  s, \dots  ) \end{array} \right.
		   \\
	\qe
If $n\geq2$, then $\mathcal K_{\xi\setminus B} $ can also be written as a covariance and so by Lemma \ref{InternalLemma} again, the quantity in the sum is independent of $\eta'$, so that 
	\eq
		 \mcal K_\xi(\underbrace{s,s^* \etc s,s^*}_{2n}) = 	 K_{\xi\setminus B} (\underbrace{s,s^* \etc s,s^*}_{2n-2}).
	\qe
By induction we hence get
	\eq
		    \mcal E\big( (ss^*)^n \big )  &=& | \mrm{NC}_2(2n)| \times (k!)^{-1}  \sum_{ \eta' \in  \mathfrak S_{k}}   u_{\eta'}.\\
	\qe
Finally, since $\phi(u_\eta) = 0$ when $\eta\neq \mrm{id}$ and $\phi(u_{\mrm{id}})=1$, we get 
	\eq
		\phi\big( (ss^*)^n \big ) & = & | \mrm{NC}_2(2n)| \times (k!)^{-1}.
	\qe
We recognize the moments of a random variable which is zero with probability $(1-k!^{-1})$, and equal to a Mar{\v{c}}henko-Pastur distribution otherwise. Hence we have proved that the $n$-th moment of the empirical singular-values distribution of $S_NS_N^*$ converges to the $n$-th moment of such a random variables. Since these moments characterize the distribution, we have proved the first item of the proposition. 

The second item is similar using the fact that 
	$$(S_N')^n = \sum_{\eps_1\etc \eps_n \in\{1, *\}} \ (S_N')^{\eps_1} \cdots (S_N')^{\eps_n}.$$

\section{Proof of preliminary lemmas of Sections \ref{Sec:OpValVar} and \ref{Sec:MainTh}}\label{Sec:ProofPrelim}

We denote by $\langle \cdot, \cdot \rangle$ the canonical scalar product of $(\mbb C^N)^k$ and by $(e_{i_1}\otimes \cdots \otimes e_{i_k})_{i_1\etc i_k\in [N]}$ its canonical basis. For any  $\eta \in \mathfrak{S}_k$, we have introduced the matrix $U_{N,\eta}$ such that for all $\mbf i, \mbf j\in [N]^n$, 
 \eq
		U_{N,\eta}(\mbf i, \mbf j) 
		& = &  \langle e_{i_1}\otimes \cdots \otimes e_{i_n}, e_{ j_{\eta(1)}}\otimes \cdots \otimes e_{j_{\eta(n)}}\rangle =\delta_{ \mbf i  ,\eta(\mbf j)},
	\qe
where we use the notation $ \eta(\mbf j) = (j_{\eta_1} \etc j_{\eta_n})$ and $\delta$ is the Kronecker symbol.

\begin{proof}[Proof of Lemma \ref{PrelimBasic1}] 
 1. The first item computes the normalized trace $\Phi_N[U_{N,\eta}]$ of the matrix. To prove it, we write 
	\eq
		\Phi_N[U_{N,\eta}] &=& N^{-k} \sum_{\mbf i \in [N]^n} \delta_{\mbf i, \eta( \mbf i)}.
	\qe
The condition $\delta_{\mbf i, \eta( \mbf i)}=1$ implies $i_{1} = i_{ \eta(1)}=i_{ \eta^2(1)}=\dots$. Hence, when $i_1\in [N]$ is given arbitrarily, the summand in the right hand side is nonzero when $i_\ell=i_1$ for all $\ell$ in the same cycle of $1$ in $\eta$. Extending the reasoning for all cycles of $\eta$ yields $\Phi_N[U_{N,\eta}]= N^{\# \eta-k}$, where $\# \eta$ is the number of cycles of $\eta$. 

2. The second item of the lemma states that the matrices $U_{N,\sigma}$ are linearly independent when $N\geq k$. Let $(a_{\eta})_{\eta \in \mathfrak{S}_k}$ be a collection of complex numbers, and let us assume that $\sum_{\eta \in \mathfrak{S}_k }a_{\eta} U_{N,\eta}=0$. When $N\ge k$, we can apply both side of the former equation to the basis vector $e_1\otimes e_2\otimes \ldots\otimes e_k$, which yields the equality $\sum_{\eta \in \mathfrak{S}_k} a_{\eta} e_{\eta(1)}\otimes \ldots \otimes e_{\eta(k)} =0$. Since $(e_{\eta(1)}\otimes \ldots \otimes e_{\eta(k)})_{\eta \in \mathfrak S_{k}}$ is a subfamily of a basis, the vectors are linearly independent, and then $a_{\eta} =0$ for all $\eta \in \mathfrak S_{k}$. Hence the matrices $(U_{N,\sigma})_{\sigma \in \mathfrak S_{k}}$ are linearly independent.

\end{proof}

We now come back to the proof of elementary properties of the maps
		$$\mcal E_N : A_N  =  \sum_{\eta \in \mathfrak{S}_k} \Phi_N\big[  A_N  U_{N,\eta}^* \big]U_{N,\eta}.$$

\begin{proof}[Proof of Lemma \ref{PrelimBasic2}]
We first prove that $\mcal E_N$ is a conditional expectation, namely Equation  \eqref{CE} holds. By linearity, it is sufficient to prove 
	$$ \mcal E_N ( U_{N,\eta_1} A_NU_{N,\eta_2}) =  U_{N,\eta_1} \mcal E_N (  A_N)U_{N,\eta_2}$$
 for all $\eta_1,\eta_2 \in \mathfrak S_k$. 
By traciality of $\Phi_N$, we have
	\eq
		\mcal E_N ( U_{N,\eta_1} A_NU_{N,\eta_2}) 
		& = & \sum_{\eta \in \mathfrak{S}_k} \Phi_N\big[  A_N U_{N,\eta_2 \eta^{-1} \eta_1} \big]U_{N,\eta}
	\qe
Using the change of variable $\eta' =  \eta_1^{-1} \eta \eta_2^{-1}$ in the sum yields
	\eq
		\mcal E_N ( U_{N,\eta_1} A_NU_{N,\eta_2}) & = & \sum_{\eta \in \mathfrak{S}_k} \Phi_N\big[  A_N U_{N,\eta^{-1}} \big]U_{N,\eta_1 \eta \eta_2}\\
		& = & U_{N,\eta_1} \mcal E_N (  A_N)U_{N,\eta_2},
	\qe
proving the first point \eqref{CE} of Lemma \ref{PrelimBasic2}. Moreover we have
	$$ \mcal E_N ( \mbb I_N)=  \sum_{\eta \in \mathfrak{S}_k} \Phi_N\big[ U_{N,\eta}^* \big]U_{N,\eta} = \sum_{\eta \in \mathfrak{S}_k}  N^{\# \eta-k} U_{N,\eta} = \mbb I_N +o(1),$$
proving the third assertion of the lemma. Finally, we have
	\eq
		\Phi_N \big[ \mcal E_N( A_N ) \big] & = &   \sum_{\eta \in \mathfrak{S}_k} \Phi_N\big[ A_N U_{N,\eta}^* \big] N^{\# \eta-k} = \Phi_N\big[ A_N ] + o(1),
	\qe
since $\# \eta=k$ only if $\eta = \mrm{id}$. This concludes the proof of Lemma \ref{PrelimBasic2}.
\end{proof}

\begin{proof}[Proof of Lemma \ref{lem:EN-cp}] 
The map $\mcal E_N$ is a normalized expectation (over the distribution of random tensors) of the map 
$$X \mapsto \sum_{\eta \in \mathfrak S_k} \Tr[ U_{N, \eta}^* X ] U_{N, \eta}.$$
Hence, it is enough to show that the linear map above is completely positive. This will follow by computing the so-called Choi matrix of the map and showing that it is positive semidefinite \cite[Theorem 2.22]{watrous2018theory}. The Choi matrix reads 
$$C = \sum_{\eta \in \mathfrak S_k} U_{N, \eta} \otimes \bar U_{N, \eta} =\sum_{\eta \in \mathfrak S_k} U_{N, \eta}^{\otimes 2} = \sum_{\eta \in \mathfrak S_k} U_{N^2, \eta} = k! P_{\mathrm{sym}},$$
where $P_{\mathrm{sym}}$ is the projection on the symmetric subspace of $(\mathbb C^{N^2})^{\otimes k}$, see \cite[Proposition 1]{harrow2013church}; this concludes the proof.
\end{proof}

\begin{proof}[Proof of Lemma \ref{MainLemma}] 
The first relation we must prove
	$$U_{N,\eta}  M_{N,\sigma}U_{N,\eta'}^* =  M_{N, (\eta \sqcup \eta') \sigma}, \quad \forall \sigma \in \mathfrak S_{2k}, \forall \eta, \eta' \in \mathfrak{S}_k,$$
follows from a simple computation of the entries: for any $\mbf i, \mbf i'\in [N]^n$,
\eq
	 U_{N,\eta}  M_{N,\sigma}U_{N,\eta'}^*(\mbf i, \mbf i') & =&  \sum_{\mbf j, \mbf j'\in [N]^n} \delta_{ \mbf i, \eta(\mbf j)} \delta_{ \mbf i', \eta(\mbf j')}  M_{N}\big( \sigma^{-1}(\mbf j, \mbf j') \big) \\
	 & = & M_{N}\big( \sigma^{-1}( \eta^{-1} \cup {\eta'}^{-1} )(\mbf i, \mbf i') \big)\\
	 	& = & M_{N, (\eta \sqcup \eta') \sigma}(\mbf i, \mbf i').\\
\qe
Moreover, we have $ U_{N,\eta}M_{N,\sigma}^*U_{N,\eta'}^*  = \big( U_{N,\eta'}M_{N,\sigma}U_{N,\eta}^*\big)^* =  M_{N, (\eta' \sqcup \eta) \sigma}^*$.
\end{proof}

\section{Proof of the convergence of the \texorpdfstring{$\mathfrak S_k$}{Sk}-covariance}\label{Sec:ProofCov}

This section is devoted to the proof of the convergence of the $\mathfrak S_k$-covariance in Theorem \ref{MainTh1}. 
We recall that for any permutation $\sigma \in \mathfrak S_{2k}$, the $\mathfrak S_{k,k} \rtimes \langle \tau \rangle$-coset of $\sigma$ is the union of all $\sigma' \in \mathfrak S_{2k}$ of the form 
	$$\sigma' =  (\eta\sqcup \eta') \sigma, \ \mrm{or \ } (\eta\sqcup \eta') \tau\sigma,  \quad \eta, \eta' \in \mathfrak S_{k}.$$

We start by proving formula \eqref{Formule1}, namely
		$$\mcal E_N( M_{N,\sigma}U_{N,\eta} M_{N,\sigma'}^*) \limN 	\left\{ \begin{array}{cc} c u_{\eta'}  & \mrm{ \ if \ }   \sigma = (\eta' \sqcup \eta)\sigma' \\ 0 & \mrm{ \ otherwise} \end{array} \right.$$

We start with the computation of the ordinary covariance
	\eq
		\lefteqn{\Phi_N\big( M_{N,\sigma}   M_{N,\sigma'}^*   \big)}\\
		& = & \esp\Big[  \frac 1 {N^k} \sum_{\mbf i\in [N]^{2k}}  M_{N}\big(\sigma^{-1}(\mbf i)\big ) \overline{M_N \big( {\sigma'}^{-1}(\mbf i)\big )}  \Big]\\
		&= & \frac 1 {N^{2k}} \sum_{\mbf i\in [N]^{2k}}  \esp\Big[  N^k M_{N}\big(\sigma^{-1}(\mbf i)\big ) \overline{M_N \big( {\sigma'}^{-1}(\mbf i)\big )}  \Big]\\
		& = &  \frac 1 {N^{2k}} \sum_{\substack {\mbf i=(i_1\etc i_{2k}) \in [N]^{2k} \\ i_p \neq i_q, \, \forall p\neq q \mrm{\, in \, } [k]}} \   \esp\Big[  N^k M_{N}\big(\sigma^{-1}(\mbf i)\big ) \overline{M_N \big( {\sigma'}^{-1}(\mbf i)\big )}  \Big] + O\big( N^{-1}\big).
	\qe
In the last line, we use the boundedness of $\esp\Big[  N^k M_{N}( \mbf i)  ) \overline{M_N  ( \mbf i')}  \Big]$ from Hypothesis \ref{TensorModel} to estimate the $O(N^{-1})$. Moreover the  entries of $M_N$ are centered and independent, so each expectation in the sum above is zero unless the entries $M_{N}\big(\sigma^{-1}(\mbf i)\big)$ and $M_N \big( {\sigma'}^{-1}(\mbf i)\big )$ are the same variable, in which case it converges to $c$. Assume that the indices of $\mbf i$ are pairwise distinct. Then these two entries are the same only if $\sigma = \sigma'$. Hence we get
	\eq
		\Phi_N\big( M_{N,\sigma}   M_{N,\sigma'}^*   \big) 
		& \limN &  c \delta_{\sigma, \sigma'}.
    \qe
	
Hence we deduce the $^*$-moments in the $M_{N,\sigma}$'s and the $U_{N,\sigma}$'s by using the first item of Lemma \ref{PrelimBasic1}:
	\eq
		\Phi_N\big( M_{N,\sigma}  U_{N,\eta}  M_{N,\sigma'}^*  U_{N,\eta'}^* \big) & = & \Phi_N\big( M_{N,\sigma}      M_{N,(\eta' \sqcup \eta) \sigma'}^*    \big)  \limN c \delta_{\sigma, (\eta' \sqcup \eta)\sigma'}.  
	\qe
	
We can deduce the limit expression for the conditional expectation
\eq
		\mcal E_N \big( M_{N,\sigma}  U_{N,\eta}  M_{N,\sigma'}^* \big) & \limN & \sum_{ \eta' \in \mathfrak S_k} c \delta_{\sigma,(\eta' \sqcup \eta) \sigma'} u_{\eta'}.
	\qe

 If $\sigma$ and $\sigma'$ are not in a same $\mathfrak S_{k,k}$-coset, there is no $\eta, \eta'$ such that $\sigma =  (\eta' \sqcup \eta) \sigma'$ and so $\mcal E_N \big( M_{N,\sigma}  U_{N,\eta}  M_{N,\sigma'}^* \big)$ converges to zero. 
	Assume now that $\sigma = (\eta_0'  \sqcup {\eta_0}) \sigma'$, for some $\eta_0, \eta_0' \in \mathfrak S_k$. Then we have
	\eq
		\mcal E_N \big( M_{N,\sigma}  U_{N,\eta}  M_{N,\sigma'}^* \big) & \limN&  \sum_{ \eta' \in \mathfrak S_k} c  \delta_{  (\eta'_0  \sqcup {\eta_0}) \sigma', (\eta' \sqcup \eta) \sigma' } u_{\eta'} \\
		& =& \sum_{ \eta' \in \mathfrak S_k}  c  \delta_{\eta, \eta_0}\delta_{\eta', \eta'_0} u_{\eta'}\\
		& = & c u_{\eta'} \quad \mrm{ where \ } \sigma = (\eta' \sqcup \eta)\sigma'.
	\qe
We have proved the announced result, formula \eqref{Formule1}.

We prove formula \eqref{Formule2}, similarly, using the traciality of $\Phi_N$
	\eq
		\mcal E_N \big( M_{N,\sigma} ^* U_{N,\eta}  M_{N,\sigma'} \big)& =& \sum_{ \eta' \in \mathfrak S_k}\Phi_N\big( M_{N,\sigma}^*  U_{N,\eta}  M_{N,\sigma'}  U_{N,\eta'}^* \big)  U_{N,\eta'} \\
		& = & \sum_{ \eta' \in \mathfrak S_k}\Phi_N\big( M_{N,\sigma'}  U_{N,{\eta'}^{-1}}  M_{N,\sigma}^*  U_{N,\eta^{-1}}^* \big)  U_{N,\eta'}\\
		& \limN & \sum_{ \eta' \in \mathfrak S_k}c \delta_{ \sigma', (\eta^{-1} \sqcup {\eta'}^{-1}) \sigma} u_{\eta'} =\sum_{ \eta' \in \mathfrak S_k} c \delta_{(\eta\sqcup {\eta'}) \sigma', \sigma} u_{\eta'}\\
		& = &  \left\{ \begin{array}{cc} c u_{\eta'}& \mrm{if \ } \sigma = (\eta \sqcup \eta')\sigma', \\
		0 & \mrm{otherwise.} \end{array} \right.   
	\qe
It remains to prove formula \eqref{Formule3}, considering $\mcal E_N( M_{N,\sigma}U_{N,\eta} M_{N,\sigma'})$. Firstly, we have with a similar reasoning
	\eq
		\lefteqn{\Phi_N\big( M_{N,\sigma}   M_{N,\sigma'}   \big)}\\
		& = & \frac 1 {N^{2k}} \sum_{\mbf i\in [N]^{2k}}  \esp\Big[  N^k  M_{N}\big(\sigma^{-1}(\mbf i)\big )  M_N \Big( {\sigma'}^{-1}\big(\tau(\mbf i)\big)\Big )  \Big]\\
		& \limN & c' \delta_{\sigma, \tau \sigma'},
	\qe
which implies that
	\eq
		\Phi_N\big( M_{N,\sigma}  U_{N,\eta}  M_{N,\sigma'}  U_{N,\eta'}^* \big) & \limN  & c' \delta_{\sigma, (\eta \sqcup \eta')\sigma'}.  
	\qe
We deduce with the same computation as before
	\eq
		\mcal E_N \big( M_{N,\sigma}  U_{N,\eta}  M_{N,\sigma'} \big)&\limN &  \left\{ \begin{array}{cc} c' u_{\eta'}& \mrm{if \ } \sigma = (\eta \sqcup \eta')\sigma', \\
		0 & \mrm{otherwise,} \end{array} \right.   
	\qe

\section{Proof of the asymptotic \texorpdfstring{$\mathfrak S_k$}{Sk}-circularity}\label{Sec:ProofCirc}

In all the section, $M_N$ denotes a random tensor that satisfies Hypothesis \ref{TensorModel}. We prove in this section that the collection of flattenings of $M_N$ converges to a $\mathfrak S_k$-circular collection. We use the method of traffics \cite{male2011traffic,cebron2016traffic}, introducing graphs notations and using combinatorial manipulations specific to this method, see below. While our ultimate goal is to compute the $\mathfrak S_k$-moments $\mcal E_N\big[M_{N,\sigma_1}^{ \eps_1} U_{N,\eta_1}\cdots M_{N,\sigma_L}^{\eps_L}U_{N,\eta_L} \big]$, the same technique as in the previous section allows us to first focus on the limit of the $^*$-moments, namely 
	\eqa\label{StarMoment}
		\Phi_N\big[M_{N,\sigma_1}^{ \eps_1} \cdots M_{N,\sigma_L}^{\eps_L} \big] =  \esp\left [ \frac 1 {N^k} \Tr \, M_{N,\sigma_1}^{ \eps_1} \cdots M_{N,\sigma_L}^{\eps_L} \right],
	\qea
where $L\geq 1$ and $\sigma_\ell \in \mathfrak S_{2k}$, $\eps_\ell\in \{1, *\}$ for all $\ell \in [L]$. Our analysis of $^*$-moments shows how the conditional expectation $\mcal E_N$ appears naturally in the large $N$ limit.

\subsection{Injective trace method for tensors}\label{Sec:InjMethod}

In this section, we first encode the $^*$-moments \eqref{StarMoment} in terms of a function on hypergraphs in \eqref{StarMomentGraph}, and then we give a general formula  in Lemma \ref{Lem:InjTrace1} for this quantity in terms of a transformation of this function that will simplify the calculations.

\begin{definition}\label{HyperTestGraph}
\begin{enumerate}
	\item We call \emph{directed $k$-hypergraph} a pair $(V,E)$ where 
	\begin{itemize}
		\item $V$ is a non-empty set, its elements are called the \emph{vertices};
		\item $E$ is a multi-set (elements can appear with multiplicity) of elements of $V^{\times 2k}$, its elements are called the \emph{hyperedges}; we often use the notation $e=(in_1 \etc in_k,out_1,\cdots, out_k)$ or $e=(\mbf {in}, \mbf {out})$, calling the k first vertices the \emph{inputs} and the k last ones the \emph{outputs}.
	\end{itemize}
	\item In this article, a \emph{$^*$-test hypergraph} is a quadruple $T=(V,E,\sigma, \eps)$ where $(V,E)$ is a directed k-hypergraph and
 $\sigma:E\to \mathfrak S_{2k}$ and $\eps:E\to \{1,*\}$ are labelling maps. With some abuse, we can think that the hyperedge $e \in E$ is associated to the matrix $M_{N,\sigma(e)}^{\eps(e)}$.
\end{enumerate}
\end{definition}

Since the domain of definition of the maps $\sigma$ and $\eps$ is the multi-set $E$, we emphasize that these functions can take different values on the different representatives of a same hyperedge, e.g $\sigma(e)$ is a multi-set (of cardinal the multiplicity of $e$) of elements of $\mathfrak S_{2k}$, that indicates the value of each representative.

\begin{definition}\label{Trace} \begin{enumerate}
\item Let $T=(V,E,\sigma,\eps)$ be a $^*$-test hypergraph. The (unnormalized) \emph{trace} of $T$ in $  M_N$ is defined by
	\eqa\label{Eq:Trace}
	\Tr \big[T(  M_N)\big] =  \sum_{ j:V \to [N]} \ \prod_{\substack{e \in E}} \ M_{N,\sigma(e)}^{\eps(e)}\big(j(e)\big),
	\qea
where the product is counted with multiplicity, and for $e = (\mbf {in}, \mbf {out})$ we have set $ j(e)= \big( j(out_1)\etc   j(out_k),  j(in_1) \etc  j(in_k)\big)$. 
\item 
Given $L\geq 1$, $\sigma_1\etc \sigma_L \in \mathfrak S_{2k}$ and $\eps_1\etc \eps_L \in \{1,*\}$, we define the $^*$-test hypergraph $T_{\sigma_1\etc \sigma_L}^{\eps_1 \etc \eps_L}=(V_L,E_L,\sigma_,\eps)$ by
\begin{itemize}
	\item $V_L= \{(1,1),\dots, (1,L), \dots, (k,1) \etc (k,L)\},$
	\item $E_L= \{ e_1, \etc e_L\}$ where 
		$$e_\ell = \big((1,\ell+1) \etc (k,\ell+1),(1,\ell) \etc (k,\ell) \big)$$
 (each edge is of multiplicity one) with the convention $(i,L+1):= (i,1)$ for all $i$ ,
	\item for all $\ell\in [L]$, we have $\sigma(e_\ell) = \sigma_\ell,$ and $\eps(e_\ell) = \eps_\ell$. 
\end{itemize}
\end{enumerate}
\end{definition}

\begin{figure}
\center \includegraphics[scale=.55]{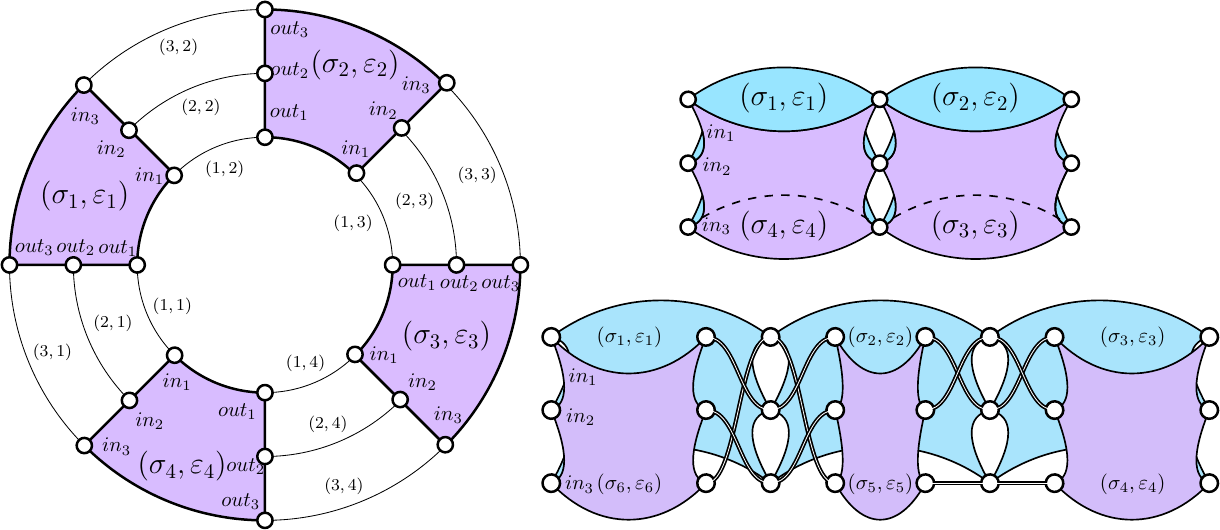}
\caption{Representations of $^*$-test hypergraphs for $k=3$. Each hyperedge is pictured as a square with its $k$ outputs represented on an edge in clockwise order, and its $k$ inputs in the opposite face in anti-clockwise order. The links between vertices represent identifications. The symbols $(\sigma_i,\epsilon_j)$ indicate the labels of the hyperedges. Left: the graph $T_{\sigma_1\etc \sigma_4}^{\eps_1\etc \eps_4}$. Right: two graphs that are commented in Section \ref{Sec:Examples}.}\label{fig_1_Test_Graph}
\end{figure}

The hypergraph $T_{\sigma_1\etc \sigma_L}^{\eps_1 \etc \eps_L}$ consists in a strip with $\ell$ successive hyperedges, see the left panel of Figure \ref{fig_1_Test_Graph}.  We then have in particular for a $^*$-test hypergraph of the form  $T= T_{\sigma_1\etc \sigma_L}^{\eps_1 \etc \eps_L}$
\eqa\label{StarMomentGraph}
	 \Phi_N\big[M_{N,\sigma_1}^{ \eps_1} \cdots M_{N,\sigma_L}^{\eps_L} \big]  =  \esp\left [ \frac 1 {N^k} \Tr  \big[T_{\sigma_1\etc \sigma_L}^{\eps_1 \etc \eps_L}(  M_N)\big]\right],
\qea
showing how we can encode the $^*$-moments of the flattenings of $M_N$ using hypergraphs.

We now define a modification of the trace of $^*$-test hypergraphs.

\begin{definition}\label{InjTrace} 
The  \emph{injective trace} of a $^*$-test hypergraph $T$ in $M_N$, denoted $\Tr^0 \big[T(  M_N)\big] $, is defined as $\Tr \big[T(  M_N)\big]$ in \eqref{Eq:Trace} but with the summation restricted to the set of injective maps $j:V \to [N]$.
\end{definition}

This functional $\Tr^0$ is related to the trace of $^*$-test hypergraphs thanks to Lemma \ref{Lem:InjTrace1}, which requires the following definition.

\begin{definition} Let $T = (V,E,\sigma, \eps)$ be a $^*$-test hypergraph. We denote by $\mcal P(V)$ the set of partitions of the vertex set $V$. Then for any $\pi \in \mcal P(V)$, we define the  $^*$-test hypergraph 
	$$ T^\pi = (V^\pi, E^\pi, \sigma^\pi, \eps^\pi),$$
called \emph{the quotient of $T$ by $\pi$}, by
\begin{itemize}
	\item $V^\pi = \pi$, i.e. a vertex is a block of $\pi$,
	\item each hyperedge $e_\ell = (i_s)_{s\in [2k]} \in E$ induces a hyperedge $e^\pi_\ell= (B_s)_{s\in [2k]}$, where $B_s$ is the block of $\pi$ containing $i_s$ for all $s\in [2k]$, and the labels of $e^\pi_\ell$ are $\sigma^\pi( e^\pi_\ell) = \sigma(e_\ell)$ and $\eps^\pi( e^\pi_\ell) = \eps(e_\ell)$. 
\end{itemize}
\end{definition}

Note that $T^\pi$ can possibly have multiple hyperedges even if $T$ does not have any, as in the top right picture of Figure \ref{fig_1_Test_Graph}. It also can have degenerated hyperedges $(B_s)_{s\in [2k]}$ where $B_s=B_{s'}$ for some $s\neq s'$ in $[2k]$ (see Figure \ref{fig_4_Algo}).

\begin{lemma}\label{Lem:InjTrace1}For any $^*$-test hypergraph $T$, we have
	$$ \Tr  \big[T(   M_N)\big] = \sum_{\pi \in \mcal P(V)}   \Tr^0  \big[T^\pi(  M_N)\big].$$
\end{lemma}

Since any $^*$-moment \eqref{StarMoment} can be written as a trace of $^*$-test hypergraph by \eqref{StarMomentGraph}, Lemma \ref{Lem:InjTrace1} implies that any $^*$-moment is a finite sum of normalized injective trace of the form \eqref{Def:Tau0}. The interest of this formulation is that the computation of injective traces is quite straightforward, and transforms the computation into a combinatorial problem.

\begin{proof}
Let $j: V \to [N]$ be a map. We denote by $\ker j$ the partition of $V$ such that $v \sim_{\ker j} v'$ whenever $j(v)=j(v')$. One can write 
	$$\Tr \big[T(   M_N)\big] =  \sum_{\pi \in \mcal P(V)} \left( \sum_{\substack{  j:V \to [N] \\ \mrm{s.t.~} \ker j = \pi}} \ \prod_{\substack{e \in E}} \ M_{N,\sigma(e)}^{\eps(e)}\big(j(e)\big)\right)$$
The lemma follows since the term in parenthesis equals $\Tr^0\big[T^\pi(  M_N)\big]$. 
\end{proof}

We also define
 	\eqa\label{Def:Tau0}
		\tau_N^0[T^\pi] & := &   \esp \Big[ \frac 1 {N^k} \Tr^0  \big[T^\pi(M_N)\big] \Big].
	\qea

\subsection{Expression of injective traces under Hypothesis \ref{TensorModel}}\label{Sec:InjExpress}

In this section we set up the definitions needed for writing an exact expression of $\tau_N^0[T^\pi] $ defined in \eqref{Def:Tau0}, for any $^*$-test hypergraph $T = (V,E,\sigma, \eps)$ and any partition $\pi$ of its vertex set $V$. We assume $N\geq |V^\pi|$ since otherwise $\tau_N^0[T^\pi] =0$.
In order to regroup terms in our computation, we need the following definition. We use as before the notation $j(e):=\big(j(v_1)\etc   j(v_{2k})\big)$ for an hyperedge $e=(v_1\etc v_{2k})$ and a function $j:V^\pi\to [N]$.

\begin{definition}\label{def:dependence} Let $T$ be a $^*$-test hypergraph and  $\pi$ be a partition of its vertex set. 
\begin{enumerate}
	\item We say that two hyperedges $e$ and $e'$ of $T^\pi$ are \emph{dependent} whenever, for any $j : V^\pi\to [N]$ injective, $M_{N,\sigma(e)}^{\eps(e)}\big(j(e)\big)$ and $M_{N,\sigma(e')}^{\eps(e')}\big(j(e')\big)$ are either the same random variable or are complex conjugate of each other. We denote by $\hat E^\pi$ the set of equivalent classes of hyperedges for the relation of dependence.
	\item For each class of dependence $\hat e$ in $\hat E^\pi$, we denote by $m(\hat e)$ and $n(\hat e)$ the number of hyperedges $e$ of $T^\pi$ in $\hat e$ such that $\eps(e)=1$ and $\eps(e)=*$ respectively. Let $x$ be distributed as the entries of $M_N$. Then we call \emph{weight} of $T^\pi$ the quantity
	\eqa\label{Def:Omega}
		\omega_N[T^\pi]  
		&  := &  \prod_{\hat e \in \hat E^\pi} N^k\times \esp\big[ x^{m(\hat e)}   \overline{x}^{n(\hat e)} \big],
	\qea
\end{enumerate}
\end{definition}
Note that by independence of the entries of $M_N$, the notion of dependence does not depend on the injective map $j$.

\begin{remark}\label{Rk:ClassDependence} Let $e$ be an hyperedge, and assume that its vertices in $T^\pi$ are pair-wise distinct. Then under Hypothesis \ref{TensorModel} two entries $M_{N,\sigma(e)}^{\eps(e)}\big(j(e)\big)$ and $M_{N,\sigma(e')}^{\eps(e')}\big(j(e')\big)$ are independent if and only if their covariance is zero. From the computation of covariances of Section \ref{Sec:ProofCov}, the hyperedges corresponding to these entries belong to a same class if
$\eps=\eps'$ and $\sigma = \sigma'$, or $\eps \neq \eps'$ and $\sigma = \tau \sigma'$. Nonetheless if the vertices of the hyperedge are not distinct this is no longer true.
\end{remark}

\begin{lemma}\label{Lem:InjExpress}For any $^*$-test hypergraph  $T$ and any a partition  $\pi$ of its vertices, we have	\eqa\label{Eq:InjExpress}
		\tau_N^0[T^\pi] & = & N^{-k- k| \hat E^\pi|} \frac{ N!  }{ (N-|V^\pi|)!}   \omega_N[T^\pi] ,
	\qea
where $\omega_N[T^\pi] $ defined in \eqref{Def:Omega} is bounded,  $\omega_N[T^\pi] =0$ if there is an equivalent class of dependence with a single element, and $\omega_N[T^\pi] \limN0$  if there is an equivalent class of dependence of cardinal different from 2.
 \end{lemma}
 
 \begin{proof}
Since $M_N$ has i.i.d. entries, the definition of the injective trace gives (with the convention $0!=1$)
\eqa
	\tau_N^0[T^\pi] & =  &  \frac{ 1 }{N^k}	 \sum_{ \substack { j:V \to [N] \\ \mrm{injective}}} \ \esp\Big[ \prod_{\substack{e \in E}} \ M_{N,\sigma(e)}^{\eps(e)}\big(j(e)\big) \Big],\nonumber \\
	& = & N^{-k} \frac{ N! }{ \times (N-|V^\pi|)!}  \esp\bigg[ \prod_{\substack{e    \in E}} \ M_{N,\sigma(e)}^{\eps(e)} \big( j(e)\big) \bigg] \label{Eq:proof:InjExpress}
	\qea 
  for any injective map $j:V^\pi\to [N]$. The value of the expectation is independent of the choice of $j$. Two edges are dependent whenever  they contribute to the same entry in Formula \eqref{Eq:proof:InjExpress}, so
  	$$\esp\bigg[ \prod_{\substack{e    \in E}} \ M_{N,\sigma(e)}^{\eps(e)} \big( j(e)\big) \bigg] = \prod_{\hat e \in \hat E^\pi} \esp\Big[   M_N\big(j(e)\big) ^{m(\hat e)}   \overline{M_N\big(j(e)\big)}^{n(\hat e)} \Big].$$
 Moreover, since the entries of $M_N$ are i.i.d., the expectations in the right hand side term above does not depend on the entry $j(e)$. It can be replaced by the entry $(\mbf 1)$ without changing the value of the expectations. This gives the expected formula. The rest of the lemma is consequence of Hypothesis \ref{TensorModel}. 
 \end{proof}

The set $\hat E^\pi$ forms a partition of the set of hyperedges that depends in an intricate way on $\pi$, $\sigma$ and $\eps$.  We will overpass the  dependence on $\sigma$ and $\eps$ by introducing another graph. It will allow, after a combinatorial analysis, to assume that $\hat E^\pi$ is a pair partition before we need to compute $\omega_N[T^\pi] $, relating our problem with the computation of the $\mathfrak S_k$-covariances studied earlier.

\subsection{Important examples}\label{Sec:Examples}

We consider the two $^*$-test hypergraphs $T^\pi$ from the rightmost pictures of Figure \ref{fig_1_Test_Graph}. We propose to compute explicitly $\tau_N^0[T^\pi]$ and its limit for each of these graphs. 

\subsubsection{A case with no twisting}

We denote by $T_1$ the $^*$-test hypergraph of the top rightmost picture of Figure \ref{fig_1_Test_Graph}, namely it is the quotient $T^\pi$ of $T=T_{\sigma_1\etc \sigma_4}^{\eps_1\etc \eps_4}$ with $k=3$ for the partition $\pi$ that identifies the $i$-th output of the second hyperedge with the $i$-th input of the third one for $i=1,2,3$. Since the second and third hyperedges of $T_1$ share the same vertices  and the same holds for the first and fourth one, each of these pair may form either a single class of dependence (consisting in two hyperedges) or two classes (consisting in single hyperedges). This depends on the values of the $\sigma_i$'s and $\eps_i$'s. If there is a class of dependence formed by a single element, then $\tau_N^0[T_1]=0$ (Lemma \ref{Lem:InjExpress}).

By Remark \ref{Rk:ClassDependence}, the second and the third hyperdeges of $T_1$ belong to a same class  if and only if one of the following dependence conditions is satisfied: $\eps_2\neq \eps_3$ and $\sigma_2 = \sigma_3$,  $\eps_2= \eps_3$ and $\sigma_2 = \tau \sigma_3$ where we recall that $\tau$ is defined as in Theorem \ref{MainTh1}

	\eq\tau(i) & : & \left\{ \begin{array}{ccc}  i \in [k] & \mapsto & i+k \in [2k]\setminus [k] \\
									i \in [2k]\setminus [k] & \mapsto & i-k \in [k].
									\end{array}\right.\qe
 Note the formal difference with the formula \ref{Rk:ClassDependence}, since we must take into account that the way the two hyperedge are identified by identifying the inputs of one with the outputs of the other one. Using the notation $[1] = 0$ and $[*]=1$, we shall write this dependence condition in short $ \tau^{[\eps_2]}\sigma_2 =\tau^{1+[\eps_3]} \sigma_3$.
The similar condition of dependence holds for the first and fourth hyperedges.  

Assume that these dependence conditions are satisfied, so $T^\pi$ has $| \hat E^\pi|=2$ classes of dependences. Note that $T_1$ has $|V_1^\pi|=9$ vertices. Hence we get
	$$N^{-k-k| \hat E^\pi|}\frac{ N!  }{  (N-|V^\pi|)!} = N^{-9}\frac{ N!  }{   (N-9)!}.$$
It remains to write the expression of $ \omega_N[T_1]$. Let $x$ be distributed as the entries of $M_N$. Denoting $x^*:= \bar x$ the usual complex conjugation, the definition of $ \omega_N$ and the above computation  finally yields the expression
	\eq
		 \tau^0_N[T_1]  & = & N^{-9}\frac{ N!  }{   (N-9)!}  \times \delta_{ \tau^{[\eps_2]}\sigma_2,\tau^{1+[\eps_3]} \sigma_3}   \delta_{ \tau^{[\eps_1]}\sigma_1,\tau^{1+[\eps_4]} \sigma_4}  \\
		 &  & \ \ \ \ \times N^k  \esp[ x^{\eps_2}\bar x^{\eps_3}] \times N^k \esp[ x^{\eps_1}\bar x^{\eps_4}],
	\qe
where $\delta$ stands for the usual Kronecker symbol. Since  $N^{-k-k| \hat E^\pi|}\frac{ N!  }{  (N-|V^\pi|)!}$ converges to one,  Hypothesis \ref{TensorModel} implies that $\tau^0_N[T_1]$ converges when $N$ tends to infinity. Denoting by $(c,c')$ the parameter of $M_N$, we get the expression of the limit using the formula
	$$\Nlim N^k \esp[ x^{\eps}\bar x^{\eps'}]  = \left\{ \begin{array}{ccc}  c & \mrm{ if } &  \eps \neq \eps', \\
		c' & \mrm{ if} & \eps=\eps'=1,\\
		\bar c' & \mrm{ if} & \eps=\eps'=*.
	\end{array} \right.$$

Finally, note that the computation of covariances of Section \ref{Sec:ProofCov} shows that
	$$\tau^0_N[T_1]  = \Phi_N( M_{N,\sigma_2}^{\eps_2} M_{N,\sigma_3}^{\eps_3} ) \times  \Phi_N( M_{N,\sigma_1}^{\eps_1} M_{N,\sigma_4}^{\eps_4} ) +o(1).$$

\subsubsection{A case with twistings}\label{Sec:ExampleTwist}

We now denote by $T_2$ the $^*$-test hypergraph of the right bottom picture of Figure \ref{fig_1_Test_Graph}, namely it is the quotient $T^\pi$ of $T=T_{\sigma_1\etc \sigma_6}^{\eps_1\etc \eps_6}$ for the partition $\pi$ that does the following identifications:
\begin{itemize}
	\item for $\eta_1$ the transposition exchanging 1 and 2, the $i$-th output of the third hyperedge is identified with the $\eta_1(i)$-th input of the fourth one, 
  
	\item for $\eta_2$ is the cycle $(1,3,2)$, the $i$-th output of the second hyperedge is identified with the $\eta_2(i)$-th input of the fifth one.
\end{itemize}
The same reasoning as before shows that there are up to three classes of dependence, the pairs of indices of possible dependent hyperedges being $ \{3,4\}, \{2,5\}$ and $\{1,6\}$. The description of the depedent classes involves now  the twisting induced by the permutations $\eta_1$ and $\eta_2$. Using the notation  $[1] = 0$ and $[*]=1$, we observe from the figure that each pair of indices $ \{3,4\}, \{2,5\}$ and $\{1,6\}$ corresponds to hyperedges of $T_2$ in a same class of dependence  whenever the following dependence conditions are satisfied: 
	$$\left\{ \begin{array}{ccc}  \tau^{[\eps_3]}\sigma_3 & =  & (\mu_1 \sqcup \mrm {id})\tau^{1+[\eps_4]}\sigma_4,\\
	 \tau^{[\eps_2]}\sigma_2 & = & (\mu_2 \sqcup \mu_1^{-1}) \tau^{1+[\eps_5]}\sigma_5 , \\
	 \tau^{[\eps_1]}\sigma_1 &=&  (\mrm{id} \sqcup \mu_2^{-1})\tau^{\eps_6}\tau^{1+[\eps_6]}\sigma_6.
	 \end{array}\right.$$

Indeed, let us consider the first formula. Since the role of the $\eps$-indices is clear from the previous example, let us assume $(\eps_3, \eps_4)= (\eps_2, \eps_5)=(\eps_1, \eps_6)=(1,*)$ for a simplification that does not affect the reasoning. Therefore the condition reads $\sigma_3 = (\mu_1 \sqcup \mrm {id}) \sigma_4$. It is the consequence of Remark \ref{Rk:ClassDependence}, the fact that by construction the $i$-th output of the third hyperedge is the $i$-th input of the fourth one, and the identification of vertices in the first item of the enumeration at the beginning of this subsection, namely of the $i$-th input of the third hyperedge is the $\eta_1(i)$-th output of the fourth one.

The second formula now reads $\sigma_2 =(\mu_2 \sqcup \mu_1^{-1})  \sigma_5 $, which follows from the second item of the above enumeration, and the fact that the first item implies that the  $i$-th output of the second hyperedge is identified by $\pi$ with  the $\eta_1^{-1}(i)$-th input of the fifth one. Similarly, the last formula reads $\sigma_1= ( \mrm{id} \sqcup \mu_2^{-1}) \sigma_6$ with same arguments, the  factor $\mrm{id}$ coming from construction and the factor $\mu_2^{-1}$ for the second itemized identification condition.

When the dependence conditions are satisfied they are $| \hat E^\pi|=3$ dependent classes of hyperedges, and $T_2$ has $|V^\pi|=12$ vertices. With the same computation of the weights as in the previous section, Lemma \ref{Lem:InjExpress} yields
\eq
    \lefteqn{ \tau^0_N[T_2]} \\
    & = & \delta_{  \tau^{[\eps_3]}\sigma_3, \mu_1\tau^{1+[\eps_4]}\sigma_4} \delta_{  \tau^{[\eps_2]}\sigma_2, \mu_2\tau^{1+[\eps_5]}\sigma_5 \mu_1^{-1}} \delta_{ \tau^{[\eps_1]}\sigma_1 , \tau^{\eps_6}\tau^{1+[\eps_6]}\sigma_6 \mu_2^{-1}}\\
    & &  \ \ \ \ \times N^{-12}\frac{ N!  }{   (N-12)!} \times \times N^k  \esp[ x^{\eps_3}\bar x^{\eps_4}] \times N^k \esp[ x^{\eps_2}\bar x^{\eps_5}] \times N^k  \esp[ x^{\eps_1}\bar x^{\eps_6}].
\qe
Hence under Hypothesis \ref{TensorModel} $\tau^0_N[T_2]$ has a limit when $N$ tends to infinity. The computation of covariances made in the dedicated section shows
	\eq
		\tau^0_N[T^\pi] &  = &  \Phi_N( M_{N,\sigma_3}^{\eps_3} M_{N,\mu_1\sigma_4}^{\eps_4} )  \Phi_N( M_{N,\sigma_2}^{\eps_2} M_{N,\mu_2\sigma_5\mu_1^{-1}}^{\eps_5} ) \\
		& &   \ \ \ \   \times   \Phi_N( M_{N,\sigma_1}^{\eps_1}   M_{N,\sigma_6\mu_2^{-1}}^{\eps_6} ) \times  +o(1).
	\qe
Lemma \ref{MainLemma} implies that 
	$$ \Phi_N( M_{N,\sigma}^{\eps} M_{N,\mu\sigma'{\mu'}^{-1}}^{\eps'} ) = \left\{ \begin{array}{ccc}  \Phi_N( M_{N,\sigma}^{\eps} U_{N,\mu} M_{N,\sigma'}U_{N,\mu'}^* ) & \mrm{if} & \eps'=1, \\
	  \Phi_N( M_{N,\sigma}^{\eps} U_{N,\mu'} M_{N,\sigma'}^*U_{N,\mu}^* )& \mrm{if} & \eps'=*.
	  \end{array}\right.
	 $$
This is a first indication of the interest of introduction the $\mathfrak S_k$-probability setting.

\subsection{Convergence of injective traces}\label{Sec:InjLimit}

In Lemma \ref{Lem:InjExpress} we have established an exact formula for $\tau_N^0[T^\pi] $ involving the normalization factor 
	$$\frac{ N! \times N^{| \hat E^\pi|} }{N^k \times (N-|V^\pi|)!}  =\big( 1 + o(1) \big) \times  N^{-k-| \hat E^\pi| + |V^\pi|},$$
where we recall that $| \hat E^\pi|$ is the number of classes of dependent edges $T^\pi$ and $ |V^\pi|$ is the number of vertices of $T^\pi$. 

In this section, we assume that $T= T_{\sigma_1\etc \sigma_L}^{\eps_1 \etc \eps_L}$ is a $^*$-test hypergraph as in Definition \ref{Trace} encoding a $^*$-moments, and assume that the class of dependent hyperedges of $T^\pi$ have at least two elements (this is the situation of interest according to Lemma \ref{Lem:InjExpress}). We prove that $-k-| \hat E^\pi| + |V^\pi|\leq 0$ and learn important properties on the case of equality. This allows us to deduce the convergence of the $\mathfrak S_k$-distribution and prepare the proof of the convergence toward a $\mathfrak S_k$-circular system.

The arguments use the following types of "simplified`` hypergraphs.

\begin{definition}\label{Skeleton}
\begin{enumerate}
	\item A \emph{simple undirected hypergraph} a pair $(V,E)$ where 
	\begin{itemize}
		\item $V$ is a non-empty set;
		\item $E$ is a set of subsets of elements of $V^\ell$ for some $1\leq \ell \leq 2k$.
	\end{itemize}

	\item For any partition $\pi$ of $V$, we set $\bar T^\pi = (V^\pi, \bar E^\pi)$, and call \emph{skeleton} of $T^\pi$, the undirected hypergraph $(V^\pi,\bar E^\pi)$ where $\bar E^\pi$ is the set of all subsets of indices $\{v_1\etc v_{2k}\}$ such that $(v_1\etc v_{2k}) \in E$. The \emph{fibre} of $\bar e \in \bar E^\pi$ is the set of all $(v_1\etc v_{2k}) \in E^\pi$ such that $\bar e = \{v_1\etc v_{2k}\}$.
\end{enumerate}
 \end{definition}

 We denote $q(\pi)= -k - k  |\bar  E^\pi | +|V^\pi|$, so that we can write
 	\eqa\label{Eq:FirstCombQuantOfInterest}
		-k-| \hat E^\pi| + |V^\pi|= q(\pi)+  k  \big( |\bar  E^\pi |-  |\hat E^\pi |  \big).
	\qea
Note that the skeleton of $T^\pi$ does not depend on the labelings $\sigma$ and $\eps$.  If two hyperedges of $T^\pi$ belong to a same class of dependence, then they are fibres of a same hyperedge in the skeleton of $T^\pi$. Hence  $|\bar  E^\pi | \leq   |\hat E^\pi |$ with equality if and only if the fibres of the hyperedges of $\bar T^\pi$  coincide with the classes of dependence of $T^\pi$. 

We  prove that $q(\pi)\leq 0$ for any quotient $T^\pi$. The idea is to consider the evolution of the combinatorial quantities under interest   \eqref{Eq:FirstCombQuantOfInterest}  for the sequence of skeletons generated by  the first $\ell$-th hyperedges of $T^\pi$ while $\ell$ increases (see Figure \ref{fig_4_Algo}). More precisely, for any $\ell=1\etc L-1$, let $S_\ell = (V_\ell,E_\ell)$ be the $^*$-test hypergraph  consisting in a open strip with $\ell$ successive hyperedges and defined as follow
\begin{itemize}
	\item $V_\ell= \{(1,1),\dots, (1,L-\ell), \dots, (k,1) \etc (k,L-\ell)\},$
	\item $E_\ell= \{ e_1, \etc e_{\ell}\}$ where 
		$$e_i = \big((1,i+1) \etc (k,i+1),(1,i) \etc (k,i) \big)$$
 (each edge is of multiplicity one),
	\item for all $i=1\etc \ell$, we have $\sigma(e_i) = \sigma_i,$ and $\eps(e_i) = \eps_i$. 
\end{itemize}
\begin{figure}
\center \includegraphics[scale=.65]{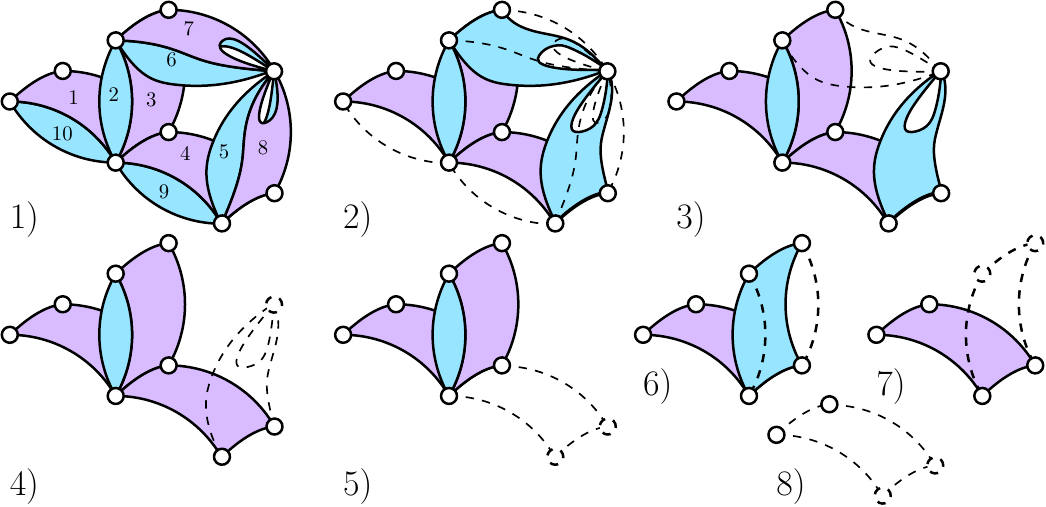}
\caption{Picture 11: a $^*$-test hypergraph $T^\pi$, where the indicate the order of the hyperedges by indices from 1 to 10. Pictures 1 to 10 represent $S_\ell^\pi$ as $\ell$ increases from 1 to 10, with the last four steps regroups for conciseness. The dot lines and vertices represent the $i$-th hyperedge while the continuous lines represent $S^\pi_{i-1}$.}\label{fig_4_Algo}
\end{figure}
Let also denote $S_L= (V_L,E_L):=T$, and let $S_0= (V_0,\emptyset)$ be the hypergraph with $k$ isolated vertices $(1,1) \etc (k,1)$ and no hyperedge. For each $\ell=0\etc L$, the partition $\pi$ induces a partition on $V_\ell$ (with a slight abuse of notation, we still denote it $\pi$) and so a quotient $S^\pi_\ell = (V^\pi_\ell,E^\pi_\ell)$ of $S_\ell$. 
For each $\ell=0\etc L$, we set
    \eqa\label{Eq:CombQuantOfInterest}
    	q(\pi,\ell) =  -k -  k  |\bar  E_\ell^\pi | +|V_\ell^\pi|,
    \qea
   where $(V_\ell^\pi,  \bar  E_\ell^\pi)$ is the skeleton of $S^\pi_\ell$ (Definition \ref{Skeleton}).

\begin{lemma}\label{LoomBound} For any $^*$-hypergraph of the form $T= T_{\sigma_1\etc \sigma_L}^{\eps_1 \etc \eps_L}$ and 
for any partition $\pi$ of its vertices, the sequence $q(\pi,\ell)_{\ell=0 \etc L}$ is non-increasing and it satisfies $q(\pi,0)\leq 0$.
\end{lemma}

The lemma clearly implies that $q(\pi) = q(\pi,L)\leq 0$, as expected.

\begin{proof} The graph $S^\pi_0$ is the quotient of $k$ isolated vertices (the outputs of the first hyperedge $e_1$) so it consists of a number $a_0 \in [k]$ of vertices. Hence we have $q(\pi,0) = - k +a_0 \leq 0$ with equality if and only if $a_0=k$, i.e. the partition $\pi$ does not identify different outputs of $e_1$.

We consider the variation  of the sequence, namely for $\ell=1\etc L$
	$$q(\pi,\ell) -q(\pi,\ell-1) =  - k \big(|\bar E_{\ell}| - |\bar E_{\ell-1}| \big) + \big( |V_\ell^\pi| - |V_{\ell-1}^\pi| \big).$$
For $\ell=1$, since the skeleton of $S_1^\pi$ has one hyperedge and $S_0^\pi$ has none, then $|\bar E_1| - |\bar E_{0}|=1$. Moreover, the vertices of $S_1$ that are not in $S_0$ (the $k$ inputs of $e_1$) can form up to $k$ new vertices in $S_1^\pi$. Setting $a_1:=|V_1^\pi| - |V_0^\pi| \in \{0 \etc k\}$, we then have 
	$q(\pi,1)  - q(\pi,0)= -k+a_1 \leq 0$
 with equality if and only if $a_1=k$, i.e. the partition $\pi$ does not identify different inputs of $e_1$, and does not identify an input of $e_1$ with a vertex consider early (at this step, these are the outputs of $e_1$).

We now come adding the other hyperedges $e_{2} \etc e_{L-1}$. For each hyperedge $e_{\ell}$, we face a choice:
\begin{enumerate}
	\item  Either it is of multiplicity one in the quotient $S^\pi_\ell$. In this case, the reasoning is the same as for $\ell=1$, namely we have $|\bar E_\ell| - |\bar E_{\ell-1}|=1$ and $a_{\ell}:=|V^\pi_\ell| -|V^\pi_{\ell-1}| \in \{0\etc k\}$, so $q(\pi, \ell) -q(\pi,\ell-1)= -k+a_{\ell-1} \leq 0$. For next section, we refer this as a \emph{growth step}.
	\item Or $e_{\ell}$ is associated with other hyperedge to form a multiple hyperedge in $S^\pi_\ell$. This implies that the skeletons of $S^\pi_{\ell}$ and $S^\pi_{\ell-1}$ are the same, and so $q(\pi, \ell) =q(\pi, \ell-1)$. For next section, we refer this as a \emph{backtrack step}.
\end{enumerate}
By induction, this proves that $q(\pi, \ell) \leq q(\pi, \ell-1)$ for all $\ell =1\etc L-1$.

Finally we add the last hyperedge. The vertex sets of $T$ and $S_{L-1}$ are the same, so $|V^\pi_L| -|V^\pi_{L-1}|=0$. Either $e_L$ is simple in $T^\pi$, in which case $|\bar E^\pi_L| -|\bar E^\pi_{L-1}|=1$ and so $q(\pi,L) = q(\pi,L-1)-k$, or it is multiple in $T^\pi$, in which case $q(\pi,L) = q(\pi,L-1)$.
\end{proof}

Now that we have proved the main result of this section, we can deduce easily the following convergence.

\begin{corollary}\label{InProofCV} The collection of flattenings of a random tensor satisfying Hypothesis \ref{TensorModel} converges in $\mathfrak S_k$-distribution.
\end{corollary}

\begin{proof} We shall prove the convergence of 
	$\mcal E_N\big[M_{N,\sigma_1}^{ \eps_1} U_{N,\eta_1}\cdots M_{N,\sigma_L}^{\eps_L}U_{N,\eta_L} \big]$
for an arbitrary choice of $L\geq 1$, $\sigma_\ell \in \mathfrak S_{2k}$, $\eps_\ell\in \{1, *\}$ and $\eta_\ell \in \mathfrak S_k$ for all $\ell \in [L]$.
The definition of $\mcal E_N$ and Lemma \ref{MainLemma} yields
\eq
	\lefteqn{\mcal E_N\big[M_{N,\sigma_1}^{ \eps_1} U_{N,\eta_1}\cdots M_{N,\sigma_L}^{\eps_L}U_{N,\eta_L} \big]}\\
	 & = & \sum_{\eta\in \mathfrak S_k} \Phi_N\big[ M_{N,\sigma_1}^{ \eps_1} U_{N,\eta_1}\cdots M_{N,\sigma_L}^{\eps_L}U_{N,\eta_L}U_{N,\eta^{-1}} \big] U_{N,\eta}\\
	& = &\sum_{\eta\in \mathfrak S_k}  \Phi_N\big[ M_{N,\tilde \sigma_1}^{ \eps_1}  \cdots M_{N,\tilde \sigma_L}^{\eps_L} \big]U_{N,\eta},
\qe
where for $\ell=1\etc L-1$, $\tilde \sigma_\ell = ( \mrm{id}\sqcup \eta_\ell^{-1}) \sigma_\ell$ if $\eps_\ell=1$ and $\tilde \sigma_\ell = ( \eta_\ell^{-1}\sqcup \mrm{id}) \sigma_\ell$ if $\eps_\ell=*$, and for $\ell=L$ we have $\tilde \sigma_L = ( \mrm{id}\sqcup \eta\eta_L^{-1}) \sigma_L$ if $\eps_L=1$ and $\tilde \sigma_L = ( \eta\eta_L^{-1}\sqcup \mrm{id}) \sigma_L$ if $\eps_L=*$.

We hence consider the $^*$-test graph $T=T_{\tilde \sigma_1\etc \tilde \sigma_L}^{\eps_1\etc \eps_L}$. We have
\eq
	\Phi_N\big[ M_{N,\tilde \sigma_1}^{ \eps_1}  \cdots M_{N,\tilde \sigma_L}^{\eps_L} \big]
		& = & \sum_{\pi\in \mcal P(V) } \tau_N^{0}\big[ T^\pi\big].
\qe
Lemmas \ref{Lem:InjExpress} and \ref{LoomBound} prove that each term in the above sum converges. Hence we get the expected convergence and a formula for the limit
\eqa\label{PreLoomFormula}
	  \Phi_N\big[ M_{N,\tilde \sigma_1}^{ \eps_1}  \cdots M_{N,\tilde \sigma_L}^{\eps_L} \big] & \limN & \sum_{ \substack{ \pi\in \mcal P(V) \mrm{ \ s.t.}\\  -k-| \hat E^\pi| + |V^\pi|=0}} \Nlim  \omega_N[T^\pi].
\qea

Let $\mcal A$ be the free $\mathfrak S_k^*$-algebra generated by a family $(m_{\sigma})_{\sigma\in \mathfrak S_k}$ with relations $u_\eta m_\sigma u_{\eta'}^*=m_{(\eta \sqcup \eta') m_\sigma}$. We equip $\mcal A$ with the linear map $\mcal E$ defined by $\mcal E\big[m_{\sigma_1}^{ \eps_1} u_{\eta_1}\cdots m_{\sigma_L}^{\eps_L}u_{\eta_L} \big] = \Nlim \mcal E_N\big[M_{N,\sigma_1}^{ \eps_1} U_{N,\eta_1}\cdots M_{N,\sigma_L}^{\eps_L}U_{N,\eta_L} \big]$. Therefore the collection of flattenings of $M_N$ converges to the family $\mbf m = (m_{\sigma})_{\sigma\in \mathfrak S_k}$ in $\mathfrak S_k^*$-distribution.
\end{proof}

\subsection{End of the proof}

We prove that the limit of the flattenings computed in the previous section is $\mathfrak S_k$-circular, using the proof of Lemma \ref{LoomBound} and manipulation explained in the second example in Section \ref{Sec:Examples}. We first state the following intermediate result, where we denote by $\mrm{NC}_2(L)$ the sets of non-crossing pair partitions of $[L]$. We recall that for a sequence $\mcal L_n$, $n\geq 1$, where $\mcal L_n$ is a $n$-linear map, we define $\mcal L_\xi$ for $\xi$ a non-crossing partition in Definition-Proposition \ref{NotationNC}.

\begin{lemma}\label{LoomLemma} Let $\mbf m=(m_{\sigma})_{\sigma\in \mathfrak S_k}$ be the limit of the collection of flattenings given in Corollary \ref{InProofCV}. We set as usual $\phi(a)$ to be the coefficient of the identity in $\mcal E(a)$, and we denote $\mcal Alg^*(\mbf m)$ the $^*$-algebra generated by $\mbf m$. There exists a collection $\mcal L_n$, $n\geq 1$ of $n$-linear forms $\mcal Alg^*(\mbf m)^n\to \mbb C$, such that for all $L\geq 3$, all $\sigma_1\etc \sigma_L$ in $\mathfrak S_{2k}$ and all $\eps_1\etc \eps_L$ in $\{1,*\}$, we have
\eq
		\lefteqn{\phi\big[ m_{\sigma_1}^{\eps_1}   \cdots  m_{\sigma_L}^{\eps_L}  \big]}\\
		& =&  \sum_{\xi \in \mrm{NC}_2(L)} \mcal L_{ \xi \setminus B} \Big[ m_{\sigma_1}^{\eps_1}    \etc m_{\sigma_{i-1}}^{\eps_{i-1}}  \mcal E \big[ m_{\sigma_{i}}^{\eps_{i}}  m_{\sigma_{i+1}}^{\eps_{i+1}} \big], m_{\sigma_{i+2}}^{\eps_{i-1}}   \etc m_{\sigma_{L}}^{\eps_{L}} \Big],
	\qe
where $B=B(\xi)=\{i,i+1\}$ denotes the first internal block of $\xi$.
\end{lemma}

\begin{proof}Let $\pi$ be a partition of the $^*$-test hypergraph $T=T_{  \sigma_1\etc   \sigma_L}^{\eps_1\etc \eps_L}$, and with the notations of Lemma \ref{Lem:InjExpress}, assume that $ -k-| \hat E^\pi| + |V^\pi|=0$. Section \ref{Sec:InjLimit} proves that necessarily the class of dependence of the graph $T^\pi$ are of cardinal 2. Given such a $\pi$, we denote by $\xi(\pi) \in \mcal P(L)$ the pair partition whose blocks are the pairs $\{\ell,\ell'\}$ of indices such that the hyperedges $e_\ell$ and $e_{\ell'}$ belong to a same class. Denoting by $\mcal P_2(L)$ the set of pair partitions of $
[L]$, we set for any $\xi \in \mcal P_2(L)$ 
\eq
	\mcal M_{ \xi } \Big[m_{\sigma_1}^{\eps_1}   \cdots  m_{\sigma_L}^{\eps_L} \Big] & := & \sum_{ \substack{ \pi\in \mcal P(V) \mrm{ \ s.t.}\\  -k-| \hat E^\pi| + |V^\pi|=0 \\ \mrm{and } \  \xi(\pi) = \xi}} \Nlim  \omega_N[T^\pi].
\qe
By our previous computation of the limit, namely \eqref{PreLoomFormula} with $\eta_1 = \etc = \eta_L= \mrm{id}$, we obtained that $ \phi\big[m_{\sigma_1}^{\eps_1}   \cdots  m_{\sigma_L}^{\eps_L} \big] = \sum_{\xi \in \mcal P_2(L)} \mcal M_{ \xi } \big[m_{\sigma_1}^{\eps_1}   \cdots  m_{\sigma_L}^{\eps_L} \big] $. Note in particular that the limit is zero if $L$ is odd.

Firstly, we shall prove that $\mcal M_{ \xi } \big[m_{\sigma_1}^{\eps_1}   \cdots  m_{\sigma_L}^{\eps_L} \big] =0$ if $\xi$ is not a non-crossing partition. Assume $L\geq 4$ is even. Recall that a pair partition $\xi$ is non-crossing if and only if there exists a internal block $B$ and the partition $\xi\setminus B$ is non-crossing. Let us prove that $\xi(\pi)$ satisfies this property when $ -k-| \hat E^\pi| + |V^\pi|=0$. Recall the proof of Lemma \ref{LoomBound}: if $ -k-| \hat E^\pi| + |V^\pi|=0$, this means that the difference $q(\pi,\ell) -q(\pi,\ell-1)$ of the combinatorial quantities defined in \eqref{Eq:CombQuantOfInterest} is zero for all $\ell=1\etc L$. 
Let $i+1$ be the first index such that the $(i+1)$-th step is a backtrack one. Then by construction $B=(i,i+1)$ is a block of $\xi(\pi)$. Now removing the block $B$ from $\xi(\pi)$ yields the partition $\tilde \xi(\pi) = \xi(\pi) \setminus B$. But the partition $\tilde \xi(\pi)$ is involved for the computation of $\phi\big[ m_{\sigma_1}^{\eps_1}   \cdots   m_{\sigma_{i-1}}^{\eps_{i-1}} \times m_{\sigma_{i+2}}^{\eps_{i+2}} \cdots m_{\sigma_L}^{\eps_L} \big] $ in the same way $\xi$ is involved in the computation of the moment under consideration. Hence by induction on $L$, we get that $\xi(\pi)$ is a non crossing partition.

It now remains to prove that for each $\xi \in \mcal P_2(L)$ we have $\mcal M_{ \xi } \big[m_{\sigma_1}^{\eps_1}   \cdots  m_{\sigma_L}^{\eps_L} \big]  = M_{ \xi \setminus B} \Big[ m_{\sigma_1}^{\eps_1}    \etc m_{\sigma_{i-1}}^{\eps_{i-1}}  \mcal E \big[ m_{\sigma_{i}}^{\eps_{i}}  m_{\sigma_{i+1}}^{\eps_{i+1}} \big], m_{\sigma_{i+2}}^{\eps_{i-1}}   \etc m_{\sigma_{L}}^{\eps_{L}} \Big]$. Again, we come back to the proof of Lemma \ref{LoomBound}. During the first backtrack step, the the $j$-th input of the hyperedge $e_{i+1}$ is identified by $\pi$ with the $\eta_\pi(j)$-th  output   of $e_{i}$, for some twisting permutation $\eta_\pi$ of $[k]$. Therefore this situation is similar to the second example in Section \ref{Sec:Examples}. \\

 The weight associated to the block $(i,i+1)\in \xi(\pi)$ with specified twisting given by a permutation $\eta_\pi$ is 
 $$\lim_{N\rightarrow \infty}N^k\esp[M_{N,\sigma_i}^{\eps_i}(\mathbf{a},\mathbf{b})(M_{N,\sigma_{i+1}}^{\eps_{i+1}}U_{N,\eta_\pi})((\mathbf{b},\mathbf{a}))],$$
 where $(\mathbf{a},\mathbf{b})$ denote the indices of the matrix elements and are all distinct (recall that both $\mathbf{a},\mathbf{b}$ are multiplets of $k$ elements and the resulting $2k$ elements are distinct). It is the weight appearing in the factorization of $\omega_N[T^\pi]$ over classes of dependence (see Formula \ref{Def:Omega} and proof of lemma \ref{Lem:InjExpress}). Moreover, since the inputs of the hyperedge associated to $M_{N,\sigma_{i+2}}^{\eps_{i+2}}$ are identified with the outputs of the one associated to  $M_{N,\sigma_{i+1}}^{\eps_{i+1}}$, the conjugate of the twisting permutation appears on input vertices of the hyperedge associated to $M_{N,\sigma_{i+2}}^{\eps_{i+2}}$ in $\omega_N[T^\pi]$ to correct for the introduction of the permutation $U_{N,\eta_\pi}$ above. This induces the identification of the outputs of $M_{N,\sigma_{i-1}}^{\eps_{i-1}}$ with the inputs of $M_{N,\sigma_{i+2}}^{\eps_{i+2}}$ through the permutation $U_{N,\eta_\pi}^*$. See Figure \ref{fig_4_Conclu} for illustration. This identification is achieved by introducing $U_{N,\eta_\pi}^*$ in front of the factorized weight 
 $$\lim_{N\rightarrow \infty}N^k\esp[M_{N,\sigma_i}^{\eps_i}(\mathbf{a},\mathbf{b})(M_{N,\sigma_{i+1}}^{\eps_{i+1}}U_{N,\eta_\pi})((\mathbf{b},\mathbf{a}))].$$
 Note the fact that  partitions $\pi$ mapping to the same partition $\xi(\pi)=\xi$ differs only by their induced twisting permutations. Therefore the weight associated to $B=(i,i+1)\in \xi$ in the expression of $\mathcal{M}_\xi$ is the sum over twisting permutations $\eta$ of the weights associated to the same block in $\xi(\pi)$ whose twisting permutation induced by $\pi$ is $\eta$. Hence, more formally the weight of the block $B$ is $$\sum_{\eta\in\mathfrak{S}_k}\lim_{N\rightarrow\infty}N^k\esp[M_{N,\sigma_i}^{\eps_i}(\mathbf{a},\mathbf{b})(M_{N,\sigma_{i+1}}^{\eps_{i+1}}U_{N,\eta})((\mathbf{b},\mathbf{a}))]U_{N,\eta}^*.$$
It is simple to recognize from the earlier workings the $\mathfrak{S}_k$-covariance $\mathcal{E}[m_{\sigma_i}^{\eps_i}m_{\sigma_{i+1}}^{\eps_{i+1}}]$ in the above formula. Therefore, we have shown that 
    \begin{multline}
    \mathcal{M}_{\xi}[m_{\sigma_1}^{\eps_1},\ldots,m_{\sigma_{i-1}}^{\eps_{i-1}},m_{\sigma_i}^{\eps_i},m_{\sigma_{i+1}}^{\eps_{i+1}},m_{\sigma_{i+2}}^{\eps_{i+2}},\ldots, m_{\sigma_L}^{\eps_L}]=\\
    \mathcal{M}_{\xi\setminus B}[m_{\sigma_1}^{\eps_1},\ldots,m_{\sigma_{i-1}}^{\eps_{i-1}}\mathcal{E}[m_{\sigma_i}^{\eps_i}m_{\sigma_{i+1}}^{\eps_{i+1}}],m_{\sigma_{i+2}}^{\eps_{i+2}},\ldots, m_{\sigma_L}^{\eps_L}].\nonumber
    \end{multline}

\begin{figure}
\center \includegraphics[scale=.6]{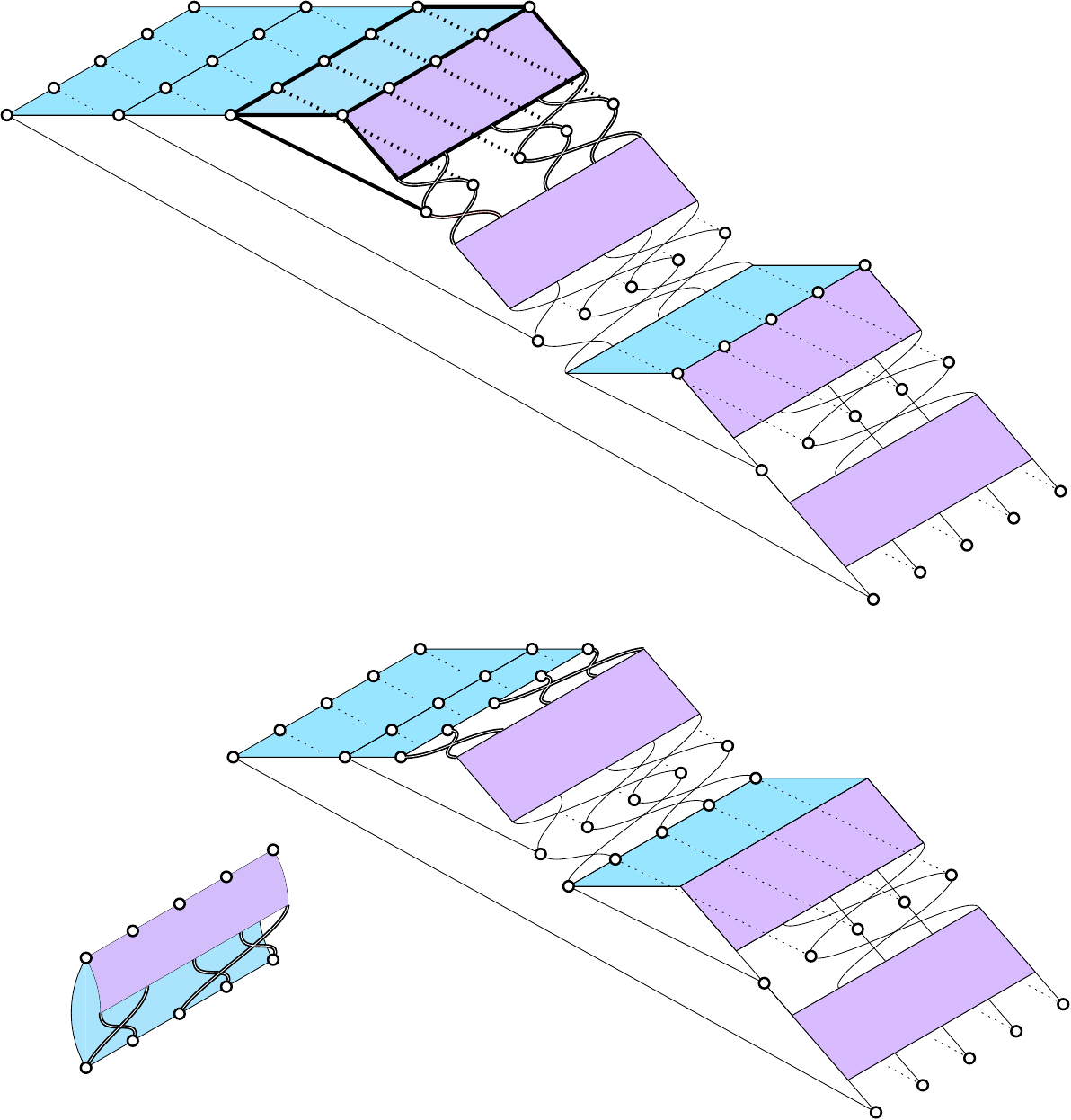}
\caption{Top: for $k=5$, we represent a quotient $T^\pi$ of a graph $T=T_{\sigma_1\etc \sigma_8}^{\eps_1\etc \eps_8}$ for a partition $\pi$ that possibly contribute in the limit $\tau^0[T^\pi]$. We have $\xi(\pi) = \big\{ \{1,8\}, \{2,5\}, \{3,4\}, \{6,7\} \big\}$, and, in cycle decompositon $\eta_1: \{(1,2), (3, 5,4)\}$, $\eta_2=  \{ (1,2,4), (3,5)\}$ and $\eta_3 = \{(1),(2,5), (3), (4) \}$. The black lines represent identifications between vertices. We emphasize the hyperedges forming the first interval block $\{3,4\}$ of $\xi(\pi)$ and their identifications with heavier lines. Bottom left: we consider the two hyperedges $\{3,4\}$ solely and identify the $j$-th input of 4 with the $\eta_1^{-1}(j)$-th output of 3. Bottom right: We represent the quotient ${T'}^{\pi'}$ of the graph $T'=T_{\sigma'_1\etc \sigma'_6}^{\eps'_1\etc \eps'_8}$ where $(\sigma'_1, \eps'_1)=(\sigma_1, \eps_1)$, $(\sigma'_2, \eps'_2)=\big( (\mrm{id} \sqcup \eta_1)\sigma_2, 1 \big)$ if $\eps_2=1$ and $(\sigma'_2, \eps'_2)=\big( ( \eta_1\sqcup \mrm{id})\sigma_2, * \big)$ if $\eps_2=*$, and $(\sigma'_i,\eps_i') = (\sigma_{i+2},\eps_{i+2}')$, for all $i\geq 3$. 
}
 \label{fig_4_Conclu}
\end{figure}

\end{proof}

We can now finish the proof of our main theorem. As in the previous section, we have
	\eq
		\mcal E\big[  m_{\sigma_1}^{\eps_1} u_{\eta_1}  \cdots  m_{\sigma_L}^{\eps_L} u_{\eta_L} \big] & = &\sum_{\eta\in \mathfrak S_k}  \Phi\big[ m_{\tilde \sigma_1}^{ \eps_1}  \cdots m_{\tilde \sigma_L}^{\eps_L} \big]u_{\eta},
\qe
where for $\ell=1\etc L-1$, $\tilde \sigma_\ell = ( \mrm{id}\sqcup \eta_\ell^{-1}) \sigma_\ell$ if $\eps_\ell=1$ and $\tilde \sigma_\ell = ( \eta_\ell^{-1}\sqcup \mrm{id}) \sigma_\ell$ if $\eps_\ell=*$, and for $\ell=L$ we have $\tilde \sigma_L = ( \mrm{id}\sqcup \eta\eta_L^{-1}) \sigma_L$ if $\eps_L=1$ and $\tilde \sigma_L = ( \eta\eta_L^{-1}\sqcup \mrm{id}) \sigma_L$ if $\eps_L=*$.
Therefore we get from Lemma \ref{LoomLemma}
	\eq
		\lefteqn{\mcal E\big[  m_{\sigma_1}^{\eps_1} u_{\eta_1}  \cdots  m_{\sigma_L}^{\eps_L} u_{\eta_L} \big] }\\
		&= &  \sum_{\xi \in \mrm{NC}_2(L)}  \sum_{\eta\in \mathfrak S_k}\mcal L_{ \xi \setminus B} \Big[ m_{\tilde \sigma_1}^{\eps_1}    \etc m_{\tilde \sigma_{i-1}}^{\eps_{i-1}}  \mcal E \big[ m_{\tilde \sigma_{i}}^{\eps_{i}}  m_{\tilde \sigma_{i+1}}^{\eps_{i+1}} \big], m_{\tilde \sigma_{i+2}}^{\eps_{i-1}}   \etc m_{\tilde \sigma_{L}}^{\eps_{L}} \Big]u_{\eta}\\
	\qe
Hence setting for each $n\geq 1$ 
	$$\tilde {\mcal K}_n( m_{\sigma_1}^{\eps_1} u_{\eta_1} \etc m_{\sigma_L}^{\eps_L} u_{\eta_L} ) = \left\{\begin{array}{ccc}  \sum_{\eta\in \mathfrak S_k} \mcal L(m_{\sigma_1}^{\eps_1}u_{\eta_1} m_{\sigma_2}^{\eps_2}u_{\eta_2}u_\eta^*)u_\eta & \mrm{if}   n=2\\
	  0 & \mrm{otherwise} \end{array}\right.
	  $$
we have proved the sequence  $\tilde K_n,n\geq 1$ satisfies the moment-cumulant relation over $\mathfrak S_k$ 
	\eq
		\lefteqn{\mcal E\big[  m_{\sigma_1}^{\eps_1} u_{\eta_1}  \cdots  m_{\sigma_L}^{\eps_L} u_{\eta_L} \big] }\\
		& = & \sum_{\xi \in \mrm{NC}_2(L)}  \tilde {\mcal K}_{ \xi \setminus B} \Big[ m_{  \sigma_1}^{\eps_1}u_{\sigma_1}    \etc m_{  \sigma_{i-1}}^{\eps_{i-1}}u_{\sigma_{-1}}  \mcal E \big[ m_{  \sigma_{i}}^{\eps_{i}}u_{\sigma_i}  m_{  \sigma_{i+1}}^{\eps_{i+1}}u_{\sigma_{i+1}} \big],\\
		& & \ \ \ \ \  m_{  \sigma_{i+2}}^{\eps_{i-1}}u_{\sigma_{i+2}}   \etc m_{  \sigma_{L}}^{\eps_{L}} u_{\sigma_L} \Big].
	\qe
By Definition-Proposition \ref{Def:Cumulants}, necessarily $\tilde K_n = K_n$ for all $n\geq 1$ and by Definition \ref{def:CircularFreeness}, the family $(m_\sigma)_{\sigma\in \mathfrak S_k}$ is $\mathfrak S_k$-circular.

\bibliographystyle{alpha}
\bibliography{biblio}

\end{document}